\documentclass[reqno]{amsart}
\usepackage{hyperref,xcolor,amssymb}
\usepackage[margin=1in]{geometry}
\usepackage[T1]{fontenc}
\usepackage{xcolor}
\usepackage[utf8]{inputenc}

\usepackage{dsfont}
\def\mcB{{\mycal B}}
\def\mcM{{\mycal M}}

\def\mcH{{\mycal H}}

\def\mcM{\mathcal{M}}

\def\mcR{{\mycal R}}

\def\mcJ{{\mycal J}}

\def\TTd{\mathbb{T}^d}
\def\iTTd{\int_{\TTd}}

\def\eps{{\varepsilon}}

\newtheorem{theorem} {\sc  Theorem\rm} [section]

\newtheorem{proposition} [theorem] {\sc  Proposition\rm}

\newtheorem{remark}[theorem]{\sc  Remark\rm}

\newcounter{marnote}

\DeclareFontFamily{OT1}{rsfs}{}
\DeclareFontShape{OT1}{rsfs}{m}{n}{ <-7> rsfs5 <7-10> rsfs7 <10-> rsfs10}{}
\DeclareMathAlphabet{\mycal}{OT1}{rsfs}{m}{n}

\def\divop{{\rm div}\,}
\def\Id{{\rm Id}}

\def\be{\begin{equation}}
\def\ee{\end{equation}}
\def\divop{{\rm div}\,}
\def\Id{{\rm Id}}

\def\mcE{{\mycal E}}

\newcommand{\R}{\mathbb{R}}

\newcommand{\N}{\mathbb{N}}

\newcommand{\T}{\mathbb{T}}

\newcommand{\mcD}{\mathcal{D}}
\newcommand{\mcW}{\mathcal{W}}

\newcommand{\fH}{\mathfrak{H}}
\newcommand{\mbE}{\mathbb{E}}

\newcommand{\bV}{\bar{V}}

\newcommand{\bW}{\bar{W}}
\newcommand{\bp}{\bar{p}}
\newcommand{\p}{\partial}

\def\be{\begin{equation}}
\def\ee{\end{equation}}
\def\bea#1\eea{\begin{align}#1\end{align}}
\def\non{\nonumber}
\def\mcR{{\mycal{R}}}

\numberwithin{equation}{section}

\begin{document}
\title{ Variational Dual Solutions
\\ for Incompressible Fluids}

\author{Amit Acharya}
\address{Amit Acharya
\hfill\break\indent
Department of Civil \& Environmental Engineering, and Center for Nonlinear Analysis
\hfill\break\indent
 Carnegie Mellon University
\hfill\break\indent
 Pittsburgh, PA 15213
\hfill\break\indent 
	{\it acharyaamit@cmu.edu}}
\author{Bianca Stroffolini}
 \address{Bianca Stroffolini
 \hfill\break\indent
 Department of Mathematics
\hfill\break\indent
University Federico II Napoli
\hfill\break\indent
Via Cintia, Napoli, 80126, Italy
\hfill\break\indent 
{\it bstroffo@unina.it}}

\author{ Arghir Zarnescu}
\address{Arghir Zarnescu
\hfill\break\indent
BCAM, Basque Center for Applied Mathematics
\hfill\break\indent
Mazarredo 14, E48009 Bilbao, Bizkaia, Spain; IKERBASQUE, Basque Foundation for Science
\hfill\break\indent
Plaza Euskadi 5, 48009 Bilbao, Bizkaia, Spain
\hfill\break\indent
``Simion Stoilow" Institute of the Romanian Academy
\hfill\break\indent
21 Calea Grivi\c{t}ei, 010702 Bucharest, Romania
\hfill\break\indent 
{\it azarnescu@bcamath.org}}

\date{\today}

\begin{abstract}

    We consider a construction proposed  in \cite{acharyaQAM}  that builds on the notion of weak solutions for incompressible fluids to provide a scheme that  generates variationally a certain type of dual solutions. If these dual solutions are regular enough one can use them  to recover standard solutions. The scheme provides a generalisation of a construction of Y. Brenier for the Euler equations. We rigorously analyze the scheme, extending the work of Y.~Brenier for Euler, and also provide an extension of it to the case of the Navier-Stokes equations. Furthermore we obtain the inviscid limit of Navier-Stokes to Euler as a $\Gamma$-limit.
\end{abstract}

\maketitle

\section{ Introduction}
An  article of Yann Brenier, \cite{Brenier-CMP}, proposed a way to reinterpret some initial value problems such as Euler equations or Burgers equation via a concave maximization problem. Indeed, the article considers the least action principle applied to the kinetic energy and using the Euler equation as a constraint. It is shown that the dual problem, which is a concave maximization problem, always admits a solution. Moreover,  sufficiently regular solutions of the Euler equation   give rise to maximizers of the dual problem.

A related viewpoint, proposed by one of us in \cite{acharyaQAM,ach2}, consists of finding critical points of a certain {\it dual }functional. This functional is constructed starting from  the weak form of the equation of interest,  the {\it primal} PDE system, with an additional auxiliary potential which is open to design in order to facilitate the solution or approximation of the given primal system. In particular, the design of the auxiliary potential admits  the inclusion of specified functions of space and time, referred to as `base states' \cite[Sec.~5]{ach2}, in order to guide solutions of the variational problem to recover those of the primal system. This feature also provides a straightforward route to global-in-time consistency of the dual formulation. As an example of such design of the auxiliary potential,  it suffices, for equations with quadratic nonlinearities like the incompressible Euler or Navier-Stokes, to choose a quadratic form in the difference of the primal fields and the corresponding base states. Formally, the critical points of the dual functional, when they are regular enough, generate weak solutions of the primal problem (i.e.~solutions of  Euler or Navier-Stokes), independent of the choice of the auxiliary potential.

This approach allows to use variational methods in studying fluids equations, providing a new, non-standard, perspective on fundamental questions in mathematical fluid mechanics. 

It will turn out that the scheme we look into extends a construction first proposed by Y. Brenier in \cite{Brenier-CMP} for Euler. As noted in  \cite{Brenier-CMP}, one has a certain type of hidden convexity of  the nonlinearity, and this feature will be fundamental for our extensions as well.

The idea of using variational methods is not new and indeed, there have been a number of other approaches such as \cite{liu2019least, Gho,OSS,GiMo,Ped}. The main driving motivation for obtaining a variational framework for solution has been the hope that this way one might be able to uncover additional structure of the equations. Indeed, perhaps the best known setting in which additional structure of the weak solutions of the equations provided useful information is  the celebrated partial regularity work of Caffarelli, Kohn and Nirenberg \cite{caffarelli1982partial} where the additional property of generalized energy inequality turns out to be the crucial element for obtaining partial regularity of (suitable) weak solutions. \par 
In the calculus of variations it is well known that minimizers(maximizers)  of energy functionals have better properties than just critical points, and thus it seems natural to explore how a variational structure might help in uncovering additional regularity of weak solutions. 
In principle, the functionals do not need to  be  differentiable, but still it is possible to infer regularity from the sole minimality assumption. Also, this allows to define solutions with less regularity than weak solutions.

\smallskip
The organisation of the paper is as follows: in the next section we describe the formal scheme as applied to Euler and Navier-Stokes, first using a divergence-free test function formulation and afterwards a formulation with general test functions. Moreover, we explain how the scheme provides an extension of Brenier's work on the Euler equation.

In Section~\ref{sec:Euler} we provide a first rigorous interpretation of the scheme for evolution problems (for statics, see \cite{sga}), allowing to obtain extremizers of a certain energy functional, which, if regular enough, would provide genuine solutions of Euler. We show first the existence of solutions and then show that the classical weak solutions can be regarded as variational dual solutions.  In Section~\ref{sec:NS} we provide an extension of the work for Euler to the case of the Navier-Stokes system, recovering the existence of solutions and showing that strong solutions provide variational dual solutions. Finally in Section 5 we show that the setting we constructed allows us to interpret the inviscid limit of Navier-Stokes to Euler as a $\Gamma$-limit.

\section{Constructing the dual problem}

\subsection {The heuristics motivating  the dual variational problem for Euler equations}
\label{sec:Eulerheuristic}

Next we explain our construction in the setting of the Euler equations and using formal manipulations. To this end
let us consider the incompressible Euler equations in a periodic domain in space
$\mathbb{T}^d=\mathbb{R}^d/\mathbb{Z}^d $ with  $d=2$ or $d=3$. For the fixed time interval $[0,T]$  the system, in classical formulation, reads as \footnote{Throughout the paper we will use the Einstein summation convention of implicit assumed summation over repeated indices.}:
\bea\label{eq:Euler}
    \partial_t V^i+ \partial_j (V^iV^j)+\partial_i p&=F^i \textrm{ on } (0,T)\times\TTd, i=1,\dots,d\non\\
       \partial_i V^i &=0\quad \textrm{ on }  (0,T)\times \TTd\non\\
           V^i(0,x) &= V^i_0(x) \quad \forall x\in \TTd, i=1,\dots,d
\eea where the unknowns are the velocity field $V=(V_1, \dots, V_d)$ and the pressure $p$, a scalar function. Here $F=(F^1,\dots, F^d)$ is a forcing term.

We  take test functions $\lambda=(\lambda_1,\dots,\lambda_d)$, $\gamma\in C^\infty ([0,T]\times \TTd)$  and such that:

\be
\partial_i \lambda_i(t,x)=0,\forall (t,x)\in [0,T]\times \TTd,  \lambda(T,x)=0,\forall x\in \TTd
\ee

Using these we consider an {\it extended}  weak formulation functional
\bea
\hat{S}_H [V,\lambda,\gamma;\bV]:=&\int_0^T\iTTd \left(V^i\partial_t \lambda_i+\frac{1}{2}V^iV^j(\partial_j\lambda_i+\partial_i\lambda_j) \right)\,dtdx+\iTTd \lambda_i(0,x)V^i_0(x),dx\non\\
&+\int_0^T\iTTd 
V^i(t,x)\partial_i\gamma(t,x)\,dtdx+\int_0^T\iTTd F^i(t,x)\lambda_i (t,x)\,dtdx \non\\
&+ \int_0^T\iTTd  H \bigg(V(t,x), \bV(t,x)\bigg)\,dtdx
\eea 
which coincides with the usual weak formulation functional but for the last  term defined in terms of the  
 auxiliary function $H(V,\bar V)$, involving the unknown velocity $V$ and  a new function $\bar V$ that we will refer to as a {\it base state}.  Both  $H$ and $\bar V$  will be clarified subsequently. For concreteness it can be convenient for now to think in terms of a quadratic function, the one that we will be using,  namely when $H$ is given by 

\be\label{def:Hquadratic0}
\mcH(V):=\frac{1}{2} a_V(V^i-\bV^i)(V^i-\bV^i)
\ee for some $a_V>0$.

We distinguish two different types of variables (leaving aside for now the $H$ and $\bar V$) :

\begin{itemize}
\item the unknowns of the original problem, namely the velocity $(V^1,\dots,V^d)$. We will refer to these  as {\it the primal variables}.
\item the test functions $\lambda,\gamma$ which we will refer to as {\it the dual variables } and we will denote them by $D:=(\lambda,\gamma)$. We will further denote by $\mathcal{D}=(\lambda,\nabla_x\lambda,\partial_t\lambda,\gamma,\nabla_x\gamma)$ the dual variables together with their derivatives.
\end{itemize}

It will be further convenient to denote:

\be
\mathcal{L}_H (V,\mcD;\bV):=V^i\partial_t \lambda_i+\frac{1}{2}V^iV^j(\partial_j\lambda_i+\partial_i\lambda_j)+V^i\partial_i\gamma+ H(V, \bV)
\ee

Thus, the extended weak formulation functional $\hat{S}_H $ can be  more simply written with this notation as:

\bea\label{def:dualEuler}
\hat{S}_H[V,\lambda,\gamma;\bV]=\int_0^T\iTTd \mathcal{L}_H (V,\mcD,\bV)(t,x)\,dtdx&+\int_{\TTd}\lambda_i(0,x)V^i_0(x)\,dx\non\\
&+\int_0^T\iTTd F^i(t,x)\lambda_i (t,x)\,dtdx 
\eea

 The main requirement we impose on the function $H$ is that it allows to express the primal variable $V$ as a function of the dual variables, namely we want to have a function $V^H:=V^H(\mathcal{D})$ such that one can solve the equation:
 
\be\label{defPHE0}
 \frac{\partial\mathcal{L}_H}{\partial V}(V^H(\mcD),\mcD,\bV)=0
\ee

Concretely, in the case of quadratic $\mathcal{H}$ considered above the equation \eqref{defPHE0} becomes:

 \bea\label{eq:VH0}
 a_V(V^{i,\mcH}-\bar V^i)+\partial_i\gamma+\partial_t \lambda_i+(\partial_j \lambda_i+\partial_i \lambda_j) V^{j,\mcH}=0
 \eea
 
which we can rewrite as:

\bea
[a_V\delta_{ij}+(\partial_j \lambda_i+\partial_i \lambda_j)]V^{j,\mcH}&=-\partial_i\gamma-\partial_t \lambda_i+a_V\bV^i
\eea 
 
 Thus, if we denote 
 
 \be\label{def:BN}
 N_{ij}(a_V,B):=a_V\delta_{ij}+\partial_i\lambda_j+\partial_j\lambda_i
 \ee 
where $B$ is the symmetric matrix whose elements are $\frac12(\partial_j \lambda_i+\partial_i \lambda_j)$, 
 we have:
 
 \bea\label{eq:VHsimple}
 V^{i,\mcH}:=(N^{-1})^{ij}\left(-\partial_j\gamma-\partial_t \lambda_j+a_V \bV^j \right) 
\eea  so $V^\mcH$ is now expressed just in terms of the dual variables .\footnote{It should be noted that in here we assumed without comment that the matrix $N$ is invertible (cf., \cite[Sec.~2; (4) in Sec.~7]{acharyaQAM}), but in the rigorous construction we will need to carefully address this issue.}

We define  now a {\it ``dual functional"} just by replacing the primal variable $V$ by the one defined in terms of the dual variable, namely $V^H(\mcD)$ specifically

\bea\label{defdualS}
S_H[\lambda,\gamma; \bV]=\int_0^T\iTTd \mathcal{L}_H (V^H(\mcD),\mcD,\bV)(t,x)\,dtdx&+\int_{\TTd}\lambda_i(0,x)V^i_0(x)\,dx\non\\
&+\int_0^T\iTTd F^i(t,x)\lambda_i (t,x)\,dtdx 
\eea

What we have achieved so far  through this procedure is to create a rather more complicated functional, namely $S_H[\cdot;\bar V]$ but which however now depends {\it only} on the dual variables and on some parameter function, $\bar V$ which we will refer to as being {\it the base state}.

We aim  to study this dual functional by means of variational methods. It turns out that, because of our construction for {\it any base state }$\bar V$ we have that {\it any sufficiently regular critical point $D$} of $S_H$,  will provide, through the transformation $V^H(\mathcal{D})$ and thanks to the property \eqref{defPHE0}  a solution of the original equation. Indeed, we have the first variation:

\bea
\delta S_H[\lambda,\gamma; \bV](\delta\lambda,\delta\gamma):=&\int_0^T\iTTd \frac{\partial \mathcal{L}_H (V^H(\mcD),\mcD,\bV)}{\partial V}\frac{\partial V^H(\mcD),\mcD,\bV)}{\partial \mcD}\delta\mcD\,dtdx+\non\\
&+\int_0^T\iTTd\frac{\partial\mathcal{L}_H (V^H(\mcD),\mcD,\bV) }{\partial\mcD}\delta\mcD+\int_{\TTd}\delta\lambda_i(0,x)V^i_0(x)\,dx\non\\
&+\int_0^T\iTTd F^i(t,x)\delta\lambda_i (t,x)\,dtdx 
\eea  where $\delta\mcD=(\delta\lambda,\delta\gamma,\delta\nabla_x\gamma,\delta\partial_t\lambda,\delta\nabla\lambda)$. Let us note that the first integral vanishes because of our requirement \eqref{defPHE0}. Furthermore as $\mathcal{L}_H$ is necessarily affine in its second variable, $\mcD$, then the second integral provides the weak formulation of the system  \eqref{eq:Euler} with $V$ replaced by $V^H(\mcD)$. 

Indeed, in the case of using the quadratic auxiliary functional $\mcH$, we have:

\bea
\frac{\partial\mathcal{L}_\mcH (V^\mcH(\mcD),\mcD,\bV)}{\partial\mcD}\delta \mcD=V^{i,\mcH}\partial_i\delta\gamma+V^{i,\mcH}\partial_t\delta\lambda_i+V^{i,\mcH }V^{j,\mcH} \partial_j\delta\lambda_i
\eea 
and correspondingly:

\bea
\delta S_\mcH[\lambda,\gamma; \bV](\delta\lambda,\delta\gamma):=&-
\int_0^T\iTTd \left[\partial_t V^{i,\mcH}(\mcD,\bV)+\partial_j\big(V^{i,\mcH}(\mcD,\bV)V^{j,\mcH}(\mcD,\bV)\big)\right]\delta\lambda_i\,dxdt\non\\
&-\int_0^T\iTTd \partial_i V^{i,\mcH}(\mcD,\bV)\delta\gamma \,dxdt+\int_{\TTd}\delta\lambda_i(0,x)\big(V^i_0-V^{i,\mcH}(\mcD,\bV)\big)\,dx\non\\
&+\int_0^T\iTTd F^i(t,x)\delta\lambda_i (t,x)\,dtdx.
\eea

 It is also worth noting from \eqref{eq:VHsimple} and the considerations immediately above that for $\bar{V}$ a weak solution to the Euler equations, $V^\mcH = \bar{V}$ for $D = 0$, and therefore $D = 0$ is a critical point of $S_H$. This is a consistency check showing that each weak solution of Euler can be recovered by mapping, via \eqref{eq:VHsimple}, a critical point of at least one dual functional designed by our scheme. These ideas have been used in \cite{KA1,KA2} to compute approximate solutions to a variety of relatively simpler problems as a preliminary test of the feasibility of the scheme.
 \begin{remark}
    A word about our perspective on the construction of the dual variational principles considered here is perhaps in order. Unlike classical variational principles, we do not consider ours as meant to shed light on the physics of the problem being dealt with, which would require any one principle in the family to recover `all' solutions (in an appropriately defined sense) of the underlying PDE model. Instead, we consider the whole family of dual variational principles as essentially a mathematical tool to facilitate the analysis, approximation, and discovery of solutions to the PDE model, by making as many choices from within the family as required in a systematic way. That this may be possible is demonstrated in computational work associated with the technique, e.g., \cite{KA2, kpa}, with the second treating a non-standard problem (a nonlinear differential-difference equation demonstrating both dispersive and solitonic behavior, posed without boundary conditions). Thus, the fact that our consistency result mentioned above  and Thm.~\ref{thm:consistency} allows consistency for global-in-time solutions with at least one principle in the family is a property we find more useful than a single variational principle providing consistency for all solutions of the PDE problem over short times.
 \end{remark}

\subsection{Connection to Brenier's construction}

We show now how the previously considered construction extends the ideas of Y. Brenier in his work \cite{Brenier-CMP}. We consider the quadratic function $\mathcal{H}$, as defined in \eqref{def:Hquadratic0} and denote $E_i=\partial_i\gamma+\partial_t\lambda_i $ and $B_{ij}:=\frac{1}{2}(\partial_i\lambda_j+\partial_j\lambda_i)$. We note that we can write the term involving the initial data in terms of these variables as:

\be\label{rel:initialEredone}
\iTTd \lambda_i(x,0)V^i_0(x),dx=-\int_0^T\iTTd  \partial_t\lambda_i(s,x)V^i_0(x)\,dsdx=-\int_0^T\iTTd E_i\cdot V^i_0\,dtdx
\ee 
where for the first equality we used the fundamental theorem of calculus in the time variable and the fact that $\lambda(x,T)=0$ while for the second equality we used that $\nabla\cdot V^0=0$.

In order to connect with the case studied by Brenier, let us consider $\bar V=F=0$ and set $a_V=1$. Then the functional \eqref{def:dualEuler} reduces to (where we use \eqref{eq:VHsimple} to express $V^H$ in terms of $N$ and $E$):
\bea
\mcE[E,B]=& \int_0^T\int_\Omega \left[-(N^{-1})^{li}E_l E_i+\frac{1}{2}((N^{-1})^{li}E_l)((N^{-1})^{mj}E_m)(N_{ij}-\delta_{ij}) \right]\,dtdx\non\\
&+ \mcH(-(N^{-1})^{li}E_l,0)-\int_0^T\int_\Omega E_iV^i_0\,dtdx\non\\
&=\int_0^T\int_\Omega \left[-(N^{-1})^{li}E_lE_i+\frac{1}{2}(N^{-1})^{li}E_l(N^{-1})^{mj}E_m N_{ij}-\frac{1}{2}(N^{-1})^{li}E_l(N^{-1})^{ji}E_j \right]\,dtdx\non\\
&+ \mcH(-(N^{-1})^{li}E_l, 0)-\int_0^T\int_\Omega E_iV^i_0\,dtdx\non\\
&=\int_0^T\int_\Omega \left[-\frac{1}{2}(\delta_{li}+2B_{li})^{-1}E_lE_i-\frac{1}{2}(N^{-1})^{li}E_l (N^{-1})^{mi}E_m \right)\,dtdx\non\\
&+ \mcH(-(N^{-1})^{li}E_l, 0)-\int_0^T\int_\Omega E_iV^i_0\,dtdx\non\\
&=-\frac{1}{2}\int_0^T\int_\Omega (\delta_{li}+2B_{li})^{-1}E_lE_i\,dtdx-\int_0^T\int_\Omega E_iV^i_0\,dtdx
\eea 

which is exactly the functional appearing in equation $(2.11)$ in Theorem $2.2$ in the paper \cite{Brenier-CMP} of Y. Brenier.

\subsection{Heuristic construction of the dual problem for the Navier-Stokes equations}

We write the incompressible  Navier-Stokes system in  first order form as follows:
\begin{subequations} \label{eq:NS}
\begin{align}
        \partial_t V^i+ \partial_j (V^iV^j)+\partial_i p&=\partial_j( W^{ij})+F^i \quad \textrm{ on }(0,T)\times\TTd, i=1,\dots,d\\
  \nu(\partial_j V^i+\partial_i V^j)&=W^{ij}\quad \textrm{ on }(0,T)\times\TTd, i,j=1,\dots,d\\
       \partial_i V^i &=0\quad \textrm{ on } (0,T)\times \TTd\\
           V^i(x,0) &= V^i_0(x) \quad x\in \TTd, i=1,\dots,d
\end{align}
\end{subequations}
and corresponding to this we consider the {\it extended} weak functional.

As for the test functions, we assume again that $\lambda$ is a smooth divergence free vector field defined in $[0,T]\times \TTd$, vanishing at $t=T$, $\gamma$ is a smooth function defined in $[0,T]\times \TTd$ and $\chi$ is a smooth  symmetric matrix field. 
\bea
\hat{S}[V,W,\gamma,\lambda,\chi;\bar V,\bar W]:=&\int_0^T\iTTd\left(V^i\partial_i\gamma+V^i\partial_t \lambda_i+\frac{1}{2}V^iV^j(\partial_j\lambda_i+\partial_i\lambda_j)- W^{ij}\partial_j \lambda_i \right)\,dtdx\non\\
&+\int_0^T\iTTd\left( 2\nu V^i\partial_j \chi_{ij}+W^{ij}\chi_{ij}\right)\,dtdx+\non\\
&+ \int_0^T\iTTd F^i(t,x)\lambda_i (t,x)\,dtdx+\iTTd  \lambda_i(x,0)V^i_0(x),dx\non\\
&+\int_0^T\iTTd\fH(V, W;\bV,\bW))\,dtdx
\eea 
where for simplicity we will limit ourselves to a quadratic auxiliary function 
\be
\fH(V,W; \bV,\bW):=\frac 12 a_V|V-\bV|^2+\frac 12 a_W|W-\bW|^2.
\ee 

Analogously to the Euler case it will be convenient to consider the integrand:
\bea
\mathcal{L}_{\fH} (V,W,\mcD;\bV,\bW):=&V^i\partial_t \lambda_i+\frac{1}{2}V^iV^j(\partial_j\lambda_i+\partial_i\lambda_j)+V^i\partial_i\gamma\non\\
&+ 2\nu V^i\partial_j \chi_{ij}+W^{ij}\chi_{ij}- W^{ij}\partial_j \lambda_i \non\\
&+\fH(V,W; \bV,\bW)
\eea where we will denote this time $\mathcal{D}=(\lambda,\nabla_x\lambda,\partial_t\lambda,\gamma,\nabla_x\gamma, \chi,\nabla_x \chi)$ the new dual variables together with their derivatives.

We have that the equation \eqref{defPHE0} becomes in this case:

\bea
\frac{\partial \mathcal{L}_{\fH} }{\partial V^i}=\partial_t \lambda_i+\partial_i \gamma+2\nu\partial_l\chi_{il}+(\partial_i\lambda_j+\partial_j\lambda_i)V^j+a_V(V^j-\bV^j)&=0\non\\
\frac{\partial \mathcal{L}_{\fH}}{\partial W^{ij}}=\chi_{ij}-\frac{1}{2}(\partial_j\lambda_i+\partial_i\lambda_j)+a_W(W^{ij}-\bar W^{ij})&=0
\eea and thus we have:

\bea\label{defVW}
V^{i,\fH}:&=-(N^{-1})^{ij}(\mbE_j-a_V\bV_j)\non\\
W^{ij,\fH}:&=\bW^{ij}+\frac{1}{a_W}( B_{ij}-\chi_{ij})
\eea where, similarly as before,  we denoted:

 \be
 B_{ij}:=\frac{1}{2}(\partial_i\lambda_j+\partial_j\lambda_i); N_{ij}(a_V):=a_V[\delta_{ij}+\frac{2}{a_V}B_{ij}]
 \ee and furthermore
 
 $$
 \mbE_i:=\partial_t \lambda_i+\partial_i \gamma+2\nu \partial_l\chi_{il}
 $$

Thus, with these notations, the dual functional becomes:

\bea\label{defdualNS}
S_\fH[\lambda,\gamma,\chi; \bV,\bW]=&\int_0^T\iTTd \mathcal{L}_{\fH} (V^{\fH}(\mcD), W^\fH(\mcD),\mcD,\bV, \bW)(t,x)\,dtdx\non\\
&+\int_{\TTd}\lambda_i(0,x)V^i_0(x)\,dx+\int_0^T\iTTd F^i(t,x)\lambda_i (t,x)\,dtdx
\eea
\bea
=&\int_0^T\iTTd V^i\mbE_i+\frac{1}{2} V^iV^j (2B_{ij}+a_V\delta_{ij})+\frac{a_V}{2} \bV^i(\bV^i-2V^i)\,dtdx\non\\
&+\int_0^T\iTTd \frac{a_W}{2}W^{ij}W^{ij}+W^{ij}(\chi_{ij}- B_{ij})\,dtdx\non\\
&+\int_0^T\iTTd\frac{a_W}{2}\bW^{ij}(\bW^{ij}-2W^{ij})\,dtdx\non\\
&-\int_0^T\iTTd V^i_0(x)\mbE_i(t,x)+2\nu\partial_l V^i_0\chi_{il}\,dtdx\non\\
&+ \int_0^T\iTTd F^i(t,x)\lambda_i (t,x)\,dtdx
 \eea 
where the penultimate integral, involving the initial data, is obtained by similar arguments as in \eqref{rel:initialEredone} in the previous section, taking into account the new definition of $\mbE$.

\subsection{Dual functional with pressure as an independent field (and quadratic auxiliary potential)}
For smooth dual variables $(\lambda,\chi,\gamma)$, $\lambda: \TTd \times (0,T) \to \R^d$ (not necessarily solenoidal/incompress-ible), $\chi:\TTd \times (0,T) \to \R^{d \times d}_{sym}$, $\gamma: \TTd \times (0,T) \to \R$ consider the extended weak form for \eqref{eq:NS} given by

\begin{align*}    \hat{S}_{\mathds{H},p}[V,W,p,\lambda,\chi,\gamma;\bV,\bW,\bp] & = \int_0^T \iTTd  V^i \p_i \gamma  +V^i \p_t \lambda_i + V^iV^j\p_j\lambda_i \,dtdx \\
    &+\int_0^T \iTTd 2\nu V^i \p_j \chi_{ij}+  
    W^{ij}\chi_{ij}-W^{ij}\p_j\lambda_i\,dtdx\\
    & +\int_0^T \iTTd  p \p_i \lambda_i +\lambda_i F^i \,dtdx  + \iTTd V^i_0(x,0) \lambda_i(x,0) \,dx  \\
    & + \frac{1}{2} \int_0^T \iTTd a_V |V - \bV|^2 + a_W |W -\bW|^2 + a_p (p - \bp)^2 \, dt dx 
\end{align*}

Then
\begingroup
\allowdisplaybreaks
\begin{align}
    \hat{S}_{\mathds{H},p} & = \int_0^T \iTTd \big((V^i - \bV^i) + \bV^i \big) ( \p_i \gamma + 2\nu \p_j \chi_{ij} + \p_t \lambda_i) \, dt dx \notag\\
    & \quad + \int_0^T \iTTd \big( (W^{ij} - \bW^{ij}) + \bW^{ij} \big) \chi_{ij}\, dt dx \notag\\
    & \quad + \int_0^T \iTTd (V^i - \bV^i) \, \p_j \lambda_i \, (V^j - \bV^j)  + (V^i \bV^j +\bV^i V^j-\bV^i \bV^j) \p_j \lambda_i \, dt dx \notag\\
    & \quad - \int_0^T \iTTd  \big( (W^{ij} - \bW^{ij}) + \bW^{ij}\big) \p_j \lambda_i \, dtdx + \int_0^T \iTTd  \big( (p -\bp) + \bp\big) \p_i \lambda_i \, dt dx \notag\\
    & \quad + \frac{1}{2} \int_0^T \iTTd a_V |V - \bV|^2 + a_W |W -\bW|^2 + a_p (p - \bp)^2 \, dt dx \notag\\
    & \quad   + \iTTd V^i_0(x,0) \lambda_i(x,0) \,dx +\int_0^T \iTTd \lambda_i F^i \, dt dx. \notag
\end{align}
\endgroup
Noting that 
\begin{align}
    ( V^i \bV^j+\bV^i V^j-\bV^i \bV^j) \p_j \lambda_i & =  \bV^j(V^i-\bV^i) \p_j \lambda_i +\bV^i \bV^j \p_j \lambda_i +\bV^i (V^j-\bV^j )\p_j \lambda_i 
\end{align}
and defining, for $\mathcal{D} := (\lambda,\nabla_x\lambda,\partial_t\lambda,\gamma,\nabla_x\gamma, \chi,\nabla_x \chi)$,
\begin{align}
    P_i (\mathcal{D};\bV) & := \p_i \gamma +2\nu \p_j \chi_{ij}+ \p_t \lambda_i+2 \bV^j B_{ij} (\mathcal{D}); \notag \\
    \qquad B_{ij} (\mathcal{D}) & := \frac{1}{2} \big( \p_j \lambda_i + \p_i \lambda_j \big); \qquad Q_{ij}(\mathcal{D}) :=  B_{ij}(\mathcal{D}) -  \chi_{ij}; \qquad r(\mathcal{D}) :=  \p_i \lambda_i \notag
\end{align}
we have:
\begingroup
\allowdisplaybreaks
\begin{align}
    \hat{S}_{\mathds{H},p} & =  \int_0^T \iTTd (V^i - \bV^i) P_i + (V^i - \bV^i) B_{ij} (V^j - \bV^j) + \frac{1}{2} a_V (V^i - \bV^i) \delta_{ij} (V^j - \bV^j)\, dt dx \notag\\
    & \quad + \int_0^T \iTTd \bV^i ( P_i - 2 \bV^j B_{ij}) +\bV^i \bV^j \p_j \lambda_i \, dt dx - \int_0^T \iTTd \lambda_i F^i \, dt dx \notag\\
    & \quad + \int_0^T \iTTd  \frac{1}{2} a_W (W^{ij} - \bW^{ij}) (W^{ij} - \bW^{ij}) -(W^{ij} - \bW^{ij}) Q_{ij}\, dt dx - \int_0^T \iTTd \bW^{ij} Q_{ij} \, dt dx \notag\\
    & \quad + \int_0^T \iTTd (p - \bp) r + \frac{1}{2} a_p (p - \bp)^2 \, dt dx + \int_0^T \iTTd \bp r \, dt dx + \iTTd V^i_0(x,0) \lambda_i(x,0) \,dx.
\end{align}
\endgroup
Let the integrand of $\hat{S}_{\mathds{H},p}$ be denoted as $\mathcal{L}_{\mathds{H},p}$.
Then the dual-to-primal mapping is obtained from the equations
\begingroup
\allowdisplaybreaks
\begin{align*}
\frac{\p \mathcal{L}_{\mathds{H},p}}{\p V^i}&: \quad a_V \left( \delta_{ij} + \frac{2}{a_V} B_{ij} (\mathcal{D}) \right) \left(V^{j,\mathds{H}} - \bV^j \right) =  - P_i(\mathcal{D};\bV) \\
\frac{\p \mathcal{L}_{\mathds{H},p}}{\p W^{ij}}&: \quad a_W \left( W^{\mathds{H},ij} - \bW^{ij} \right) = Q_{ij}(\mathcal{D}) \\
\frac{\p \mathcal{L}_{\mathds{H},p}}{\p p}&: \quad a_p \left(p^\mathds{H} - \bp\right) = - r(\mathcal{D}).
\end{align*}
\endgroup
Defining
\[
\mathbb{N}_{ij}(\mathcal{D}):= \delta_{ij} + \frac{2}{a_V} B_{ij}(\mathcal{D}),
\]
the dual functional, $S_{\mathds{H},p}[\lambda, \chi, \gamma]$, is obtained by substituting $(V, W, p)$ by $\left( V^\mathds{H}, W^\mathds{H}, p^\mathds{H} \right)$ in $\hat{S}_{\mathds{H},p}$. Thus,
\begingroup
\allowdisplaybreaks
\begin{align*}
  S_{\mathds{H},p}[\lambda, \chi, \gamma&;\bV,\bW,\bp]  =   \\
  & \int_0^T \iTTd \left(- 1 + \frac{1}{2} \right) a_V (V^{i,\mathds{H}}(\mathcal{D};\bV) - \bV^i) \mathbb{N}_{ij}(\mathcal{D}) (V^{j,\mathds{H}}(\mathcal{D};\bV) - \bV^j) \, dt dx \\
  &  + \int_0^T \iTTd \left(- 1 + \frac{1}{2} \right) a_W (W^{\mathds{H},ij}(\mathcal{D}) - \bW^{ij})(W^{\mathds{H},ij}(\mathcal{D}) - \bW^{ij}) \, dt dx\\
  & + \int_0^T \iTTd \left(- 1 + \frac{1}{2} \right) a_p (p^\mathds{H}(\mathcal{D}) - \bp)^2 \, dt dx \\
  & + \int_0^T \iTTd \bV^i ( P_i(\mathcal{D};\bV) -2 \bV^j B_{ij}(\mathcal{D})) +\bV^i \bV^j B_{ij}(\mathcal{D}) -\bW^{ij} Q_{ij}(\mathcal{D}) + \bp r(\mathcal{D}) \, dt dx \\
  & +\iTTd V^i_0(x,0) \lambda_i(x,0) \,dx - \int_0^T \iTTd \lambda_i F^i \, dt dx \\
  & \qquad \qquad = \\
  & - \frac{1}{2} \int_0^T \iTTd \frac{1}{a_V} P_i(\mathcal{D};\bV) (\mathbb{N}^{-1})^{ij}(\mathcal{D}) P_j(\mathcal{D};\bV) + \frac{1}{a_W} Q_{ij}(\mathcal{D})Q_{ij}(\mathcal{D}) + \frac{1}{a_p} r^2(\mathcal{D}) \, dx dt\\
  & + \int_0^T \iTTd \bV^i P_i(\mathcal{D};\bV) - \bV^i \bV^j B_{ij}(\mathcal{D}) - \bW^{ij} Q_{ij}(\mathcal{D}) + \bp r(\mathcal{D}) \, dt dx \\
  &  +\iTTd V^i_0(x,0) \lambda_i(x,0) \,dx - \int_0^T \iTTd \lambda_i F^i \, dt dx. \\
\end{align*}
\endgroup

It can be shown that the dual system of equations, corresponding to the primal system \eqref{eq:NS} and a `shifted' quadratic for the auxiliary function: 
\[
\mathds{H}(V,W,p; \bar{V},\bar{W},\bar{p}) = \frac{1}{2} \left( a_V |V - \bV|^2 + a_W |W -\bW|^2 + a_p (p - \bp)^2 \right)
\]
including the pressure field, is locally degenerate elliptic in space-time (about an appropriate base state). The system \eqref{eq:NS} may be posed in the form
\begin{equation}\label{eq:deg_ell}
    \sum_\alpha \p_\alpha (\mathcal{F}_{\Gamma \alpha}(U)) + G_\Gamma (U,x) = 0, \qquad \Gamma = 1, \ldots, N^*; \qquad \alpha = 0,\ldots, d,
\end{equation}
where $ N^* = d + d^2 + 1$, $x = (t,x_1,\ldots, x_d)$, $U = (V,W,p)$ and
\begingroup
\allowdisplaybreaks
\begin{align*}
    & \mathcal{F}_{\Gamma 0} = V_\Gamma; \qquad \mathcal{F}_{\Gamma \alpha} = V_\Gamma V_\alpha + p \delta_{\Gamma \alpha} -  W_{\Gamma \alpha} ; \qquad G_\Gamma = - F_\Gamma; \\
    & \qquad \qquad \qquad \qquad \mbox{for} \quad  \Gamma = 1, \ldots, d; \quad \alpha = 1, \ldots, d.\\
    & \mathcal{F}_{\Gamma 0} = 0; \qquad \mathcal{F}_{\Gamma k} = \frac{1}{2} V_i; \qquad \mathcal{F}_{\Gamma i} = \frac{1}{2} V_k; \qquad \mathcal{F}_{\Gamma s}  = 0 \quad \mbox{for} \  s \neq i, s \neq k; \qquad G_\Gamma = - W_{ik}; \\
    & \qquad \qquad \qquad \qquad \mbox{for} \quad i = 1,\ldots,d; \quad k = 1,\ldots,d; \qquad \Gamma := d + (k-1)d + i.\\
    & \mathcal{F}_{\Gamma 0} = 0; \qquad \mathcal{F}_{\Gamma \alpha} = V_\alpha; \qquad G_\Gamma = 0\\
    & \qquad \qquad \qquad \qquad \mbox{for} \quad \Gamma = d + d^2 + 1; \qquad \alpha = 0,\ldots, d.
\end{align*}
\endgroup
This representation puts it in the form examined in \cite[Sec.~3]{acharya2023hidden} to assess the ellipticity of the dual E-L equations corresponding to a primal system of the form \eqref{eq:deg_ell}. Following that argument, the dual PDE system corresponding to \eqref{eq:NS}, with the dual field $D$ comprising $N^*$ components, can be  shown to be locally degenerate elliptic about the dual state $D = 0$.

\section{Variational Dual Solutions for Euler}
\label{sec:Euler}

In order to make rigorous the previous heuristical constructions, let us begin by noting that one of the main benefits of working with a variational formulation is related to the possibility of obtaining solutions as extremal points. In calculus of variations these typically are  minimizers but it will turn out that thanks to the structure of the problem we consider it is convenient to work in fact with both minimizers and maximizers.

The first natural thing to attempt is to solve the problem defining the relationship between the original and the dual problem, namely \eqref{defPHE0}. We will do this as a minimisation problem, namely  we will look into 

\bea \textrm{inf}_V &\int_0^T\iTTd \left(V^i\partial_t \lambda_i+\frac{1}{2}V^iV^j(\partial_j\lambda_i+\partial_i\lambda_j) \right)\,dtdx+\iTTd \lambda_i(x,0)V^i_0(x),dx\non\\
&+\int_0^T\iTTd
V^i(t,x)\partial_i\gamma(t,x)\,dtdx+\int_0^T\iTTd F^i(t,x)\lambda_i (t,x)\,dtdx \non\\
&+\int_0^T\iTTd  \mcH\bigg(V(t,x), \bV(t,x)\bigg)\,dtdx
\eea where $\mcH$ is the quadratic function defined in \eqref{def:Hquadratic0}.

It is clear that the integrals will play no role here and in fact the minimisation is pointwise and the minimising $V^\mcH$ will satisfy pointwise the equation \eqref{eq:VH0}. We will then aim to solve the dual problem, by finding  an extremal point of the dual energy functional (in which $V$ appears as a function of $\lambda,\gamma$). It turns out that in fact the natural structure of the problem, of a concave functional type, requires to look into a  maximimisation problem in the dual variable. 

Thus we are lead, by building on the ideas from \cite{Brenier-CMP}, to considering the following $\textrm{sup-inf}$ problem:

\bea
&\sup_{(\lambda,\gamma)}\inf_V \int_0^T\iTTd \frac{1}{2} V^iV^j(\partial_j\lambda_i+\partial_i\lambda_j+a_V\delta_{ij})+V^i\left(\partial_t \lambda_i+
\partial_i\gamma -a_V \bar{V}_i\right)\,dtdx+\non\\
&+\int_0^T\iTTd \frac{1}{2}a_V \bV^i(t,x)\bV_i(t,x)+\frac{1}{T}\lambda_i(x,0)V^i_0(x)+F^i(t,x)\lambda_i(t,x)\,dtdx
\eea

Let us first observe that in order to have that the infimum in $V$ exists,  we need to have pointwise the following positive semidefinitness restriction on the matrix $a_V\Id+2B$, where $B_{ij}:=\frac{1}{2}(\partial_i\lambda_j+\partial_j\lambda_i)$:

\be
a_V \Id+2B\ge 0
\ee which  implies in particular, denoting by $f_i, i=1,\dots,d$ the eigenvalues of $B$, that 

$$0\le a_V+2f_i$$ Using this positivity together with the fact that $\sum_{j=1}^d f_j=0$ (due to $\divop\lambda=\textrm{tr}B=0$) gives us

$$a_V+2f_i\le \sum_{j=1}^d (a_v+2f_j)=da_V $$ hence

\be
-\frac{a_V}{2}\le f_i\le \frac{d-1}{2}a_V, \forall i=1,\dots,d
\ee which implies in particular

\be
|\nabla^{sym}\lambda|\le \sqrt{\sum_{j=1}^d | f_j|^2}\le \sqrt{d}\frac{(d-1)}{2}a_V, \forall i,j=1,\dots,d
\ee where we denote $\nabla^{sym}\lambda:=\frac{1}{2}(\nabla\lambda +(\nabla \lambda)^T)$ and in here $|A|=\sqrt{\sum_{i,j=1}^d a_{ij}^2}$ denotes the Frobenius norm of the symmetric matrix $A=(a_{ij})_{i,j=1,\dots,d}$.

Thus, it is natural to require that $\nabla^{sym}\lambda\in L^\infty(0,T;L^\infty)(\TTd))$. Furthermore, in order to make sense of the term $\int_0^T \iTTd V_i\partial_t\lambda_i$ for $V\in L^2(0,T; L^2(\TTd))$ we need to assume $\partial_t\lambda\in L^2(0,T; L^2(\TTd))$. Similarly because of the term  $\int_0^T\iTTd V_i\partial_i\gamma \,dtdx $ one needs to take $\nabla\gamma\in L^2(0,T,L^2(\TTd))$. 

Thus it will be convenient to work with functions in the space:

\bea\label{def:mathfrakR}
\mcR:=\{(\lambda,\gamma);&\lambda \in H^1_T(L^2); \nabla^{sym}\lambda\in L^\infty_T(L^\infty),\rm{div}\lambda=0; \lambda(T,\cdot)=0,\, \rm{ a.e.} \,x\in\TTd,\non\\
&\gamma\in L^2_T(H^1),\int_{\TTd}\gamma(t,x)dx=0,\rm{ a.e.}t\in [0,T]\}
\eea

Then we  have the following analogue of Brenier's Theorem $2.2$ namely:

\begin{theorem} 
\label{existenceEuler}
Let $V_0\in L^2(\TTd,\mathbb{R}^d)$ be 
a divergence-free vector field 
of zero spatial mean. Furthermore let $F,\bV\in L^2_T(L^2(\TTd))$. Then the sup-inf problem 
\bea\label{def:infsuprigorous}
\mcJ_{E}[\bV](V_0)=& \sup_{(\lambda,\gamma)\in\mathfrak{R}}\inf_V \Bigg(  \int_0^T\iTTd \frac{1}{2} V^iV^j(\partial_j\lambda_i+\partial_i\lambda_j+a_V\delta_{ij})\,dtdx\non\\
&+\int_0^T\iTTd V^i\left(\partial_t \lambda_i+
\partial_i\gamma -a_V \bar{V}_i\right)\,dtdx+\non\\
&+\int_0^T\iTTd \frac{1}{2}a_V (\bV^i\bV_i+F^i\lambda_i)(t,x)+\frac{1}{T}\lambda_i(x,0)V^i_0(x)\,dtdx \Bigg) 
\eea always admits a solution $(\lambda,\gamma)\in\mcR$ where $\mcR$ is defined in \eqref{def:mathfrakR}.

\end{theorem}

\begin{proof}
We will  follow closely the ideas of the proof of Theorem 2.2 in \cite{Brenier-CMP}, essentially adapting it in terms of the variables $\lambda,\gamma$ that allow to treat the presence of the forcing term and also taking into account the presence of the base state $\bV$. Let us  denote
\bea
E_i:=\partial_t \lambda_i+
\partial_i\gamma, \quad B_{ij}:=\frac{1}{2}(\partial_i\lambda_j+\partial_j\lambda_i)
\eea with $E_i$ similar as in \cite{Brenier-CMP}.

Then the fact that $(\lambda,\gamma)\in\mcR$ implies $E\in L^2_T(L^2)$ and $B\in L^\infty_T(L^\infty)$.

Out of the discussion before the theorem we have:
\begin{equation}\label{semidef}
-a_V\mathbb{I}_d \le 2B\le a_V(d-1)\mathbb{I}_d,
\end{equation}
which provides an  apriori $L^\infty$ bound for $B$ (since $B$ is valued in the set of symmetric
matrices).

Furthermore we note that we can rewrite a bit the term involving the initial data $V_0$ (noting that  we have for a.e. $x\in\T^d$ that  $\lambda(0,x)=-\int_0^T \partial_t\lambda(s,x)dx$ since $\lambda(T,x)=0$):

\bea
\iTTd V^i_0(x)\lambda_i(0,x)\,dx&=-\int_0^T\iTTd V^i_0(x)\partial_t\lambda_i(s,x)\,dsdx\non\\
&=-\int_0^T\iTTd V^0_i(x)E_i(t,x)\,dtdx\non
\eea (where for the last equality we used that $\partial_i V^i_0=0$).

Moreover, we can deduce an $L^2$ bound on $\lambda$. Indeed, using that $\lambda(T,\cdot)\equiv 0$ we have out of the definition of $E$:

\be
\lambda_i(x,t)=-\int_t^T \Bigg(E_i+\partial_i(-\Delta)^{-1}(\nabla\cdot E)\Bigg) ds
\non
\ee hence

\begin{align}
\|\lambda\|_{L^2_T(L^2)}^2&=\sum_{i=1}^3\iTTd\int_0^T|\int_t^T \Bigg(E_i+\partial_i(-\Delta)^{-1}(\nabla\cdot E)\Bigg) ds|^2\,dtdx\non\\
&\le\sum_{i=1}^3 \iTTd\int_0^T \sqrt{(T-t)}\Bigg(\int_t^T |E_i+\partial_i(-\Delta)^{-1}(\nabla\cdot E)|^2 ds\Bigg)\,dtdx\\
&=\sum_{i=1}^3\iTTd\int_0^T -\frac{2}{3}\frac{d}{dt}(T-t)^{\frac{3}{2}}\Bigg(\int_t^T |E_i+\partial_i(-\Delta)^{-1}(\nabla\cdot E)|^2 ds\Bigg)\,dtdx\non\\
&=\sum_{i=1}^3\iTTd\int_0^T \frac{2}{3}(T^{\frac{3}{2}}-(T-t)^{\frac{3}{2}})|E_i+\partial_i(-\Delta)^{-1}(\nabla\cdot E)|^2\, dtdx\non\\
&\le C(T)\|E\|_{L^2_T(L^2)^2}^2\label{est:lambdaL2}
\end{align}

We notice that we cannot invert the matrix $N(a_V,B):= a_V\Id +2B$, since it is only positive semidefinite, by \eqref{semidef}. However, similarly as in \cite{Brenier-CMP}, we can use a duality argument and write pointwise in  $(t,x)$:
$$
E\cdot(a_V\mathbb{I}_d+2B)^{-1}\cdot E
=\sup_{M,Z} \;\;\Big\{2E_iZ^i-(a_V\delta_{ij}+2B_{ij})M^{ij}\Big\}
$$
where $M$ and $Z$ are respectively $d\times d$ symmetric
matrices and vectors in $\mathbb{R}^d$ subject to
\begin{equation}\label{matrineq}
Z\otimes Z\le M,
\end{equation}
in the sense of symmetric matrices.
This allows us to define the following object:
\begin{equation}\label{K00}
\mcM [E,B]=\sup_{M\ge Z\otimes Z} \;\;\frac{1}{2}\int_0^T\int_{\TTd}
2 E_iZ^i-(a_V\delta_{ij}+2B_{ij})M^{ij}
\in [-\infty,0],
\end{equation}
where the supremum is performed over all pairs $(Z,M)$ 
of continuous functions on
$[0,T]\times \TTd$, respectively valued
in $\mathbb{R}^d$ and in the set of 
symmetric matrices $M$.
Notice that definition (\ref{K00}) makes sense already as $(E,B)$  belong to $L^2\times L^\infty$. Thus $\mcM$ is a lower semi-continuous  function of $(E,B)$ valued in $[0,+\infty]$.

We can now minimize in $V$ and observe that using the previously defined $\mcM$ we have:
\begin{align}\label{ivpdual00}
\mcJ_{E}[\bV](V_0)=&\sup_{(\lambda,\gamma)\in\mathfrak{R}}
\Bigg(-\mcM[E-a_V \bV,B]+\int_0^T\iTTd \frac{1}{2}a_V \bV^i\bV^i+F^i\lambda_i-V^i_0 E_i\,dtdx
\Bigg), 
\end{align}
where, formally, we can write 
\begin{equation}\label{defM}
\mcM[E,B]=\int_0^T\iTTd E\cdot(a_V\mathbb{I}_d+2B)^{-1}\cdot E \,dtdx.
\end{equation}

This immediately implies $\mcJ_{E}[\bV](V_0)\ge 0$
(just by taking $\lambda=\gamma=0$, hence $E=B=0$). 
\\

Next, because of the lower bound $N(a_V,B)^{-1}=(a_V\Id+2B)^{-1}\ge (da_V)^{-1}$, we get an $L^2$ a priori bound for $E$. Indeed,
by definition (\eqref{defM}) of $\mcM$, we have:
\be\label{equicoercivityK}
\frac{1}{2da_V}\int_0^T\iTTd |E-a_V\bV|^2\le \mcM [E-a_V\bV,B].
\ee
So, for any $\varepsilon-$maximizer $(E,B)\in L^2\times L^\infty$ of (\ref{ivpdual00}), we get
\bea
0&\le  \mcJ_{E}[\bV](V_0)
\le \varepsilon-\mcM [E-a_V\bV,B]+\int_0^T\iTTd \frac{1}{2}a_V \bV^i\bV_i\,dtdx+\int_0^T\iTTd (F^i\lambda_i)(t,x)\,dtdx\non\\
&-\int_0^T\iTTd V^i_0E_i \,dtdx\le \varepsilon-\frac{1}{4da_V}\int_0^T\iTTd 2|E-a_V\bV|^2+4da_V V_0\cdot (E-a_V\bV)\non\\
&+\int_0^T\iTTd \frac{a_V}{2}|\bV|^2-a_V\bV\cdot V^0+2C(T)da_V F^2+\frac{\lambda^2}{C(T)8da_V}\,dtdx\non\\
&\le \varepsilon-\frac{1}{4da_V}\int_0^T\iTTd |E-a_V\bV|^2\,dtdx+\int_0^T\iTTd \frac{a_V}{2}|\bV|^2-a_V\bV\cdot V^0+2C(T)da_V F^2\,dtdx\non\\
&+\int_0^T\iTTd da_V|V^0|^2\,dtdx+\frac{1}{8da_V}\int_0^T\iTTd |E-a_V\bV|^2+a_V^2  |\bV|^2\,dtdx\label{est:mcJbVV0}
\eea (where $C(T)$ is the constant from \eqref{est:lambdaL2})

which provides the a priori bound :
\bea\label{bound-Echi}
\|E-a_V\bV\|_{L^2_T(L^2)}^2&\le 8da_V(\varepsilon+(d+\frac 12)a_V\|V_0\|_{L^2}^2)+(a_V^2+8da_V^2)\|\bV\|_{L^2_T(L^2)}\\
&+ 16 C(T)d^2a_V^2 \|F^2\|_{L^2_T(L^2)}^2
\eea hence

\bea
\|E\|_{L^2_T(L^2)}^2&\le 2\|E-a_V\bV\|_{L^2_T(L^2)}^2+2a_V^2\|\bV\|^2\\ &\le 16da_V(\varepsilon+(d+\frac 12)a_V\|V_0\|_{L^2}^2)+2(2a_V^2+8da_V^2)\|\bV\|_{L^2_T(L^2)}\\
&+ 32 C(T)d^2a_V^2 \|F^2\|_{L^2_T(L^2)}^2
\eea

We also deduce, out of \eqref{est:mcJbVV0}, by passing $\varepsilon\to 0$ and dropping the negative terms on the right hand side \\
$$
0\le \mcJ_{E}[\bV](V_0)\le (\frac{a_V}{8d}+\frac{a_V}{2})\|\bV\|_{L^2_T(L^2)}+da_V
\|V^0\|_{L^2}+\int_0^T\iTTd 2 C(T)da_V F^2-a_V\bV\cdot V^0\,dtdx
$$

By definition \eqref{K00}, $\mcM$ is lower semi-continuous with respect to the weak-* topology of 
$ L^2\times L^\infty$, while
$$
(\lambda,\gamma)\to \int_0^T\iTTd \frac{1}{2}a_V \bV^i\bV_i+F^i\lambda_i-V^i_0(E_i+a_V\bV_i)\,dtdx
$$ 
are continuous (since $V_0,F,\bV$ are given in $L^2$) 
Thus, we conclude that the maximization problem
(\ref{ivpdual00}) 
always has at least an optimal solution with $(E,B)\in L^2\times L^\infty$ (since its 
$\varepsilon-$maximizers stay confined in a fixed ball (and therefore a weak-* compact subset)  of $L^2\times L^\infty$, as $\varepsilon$ goes to zero). We note that $(E,B)\in L^2\times L^\infty \Leftrightarrow (\lambda,\gamma)\in \mcR$ hence the solution is well-defined in terms of our original variables $\lambda,\gamma$.

\end{proof}
\begin{remark} We will refer to a pair $(\gamma,\lambda)\in\mathfrak{R}$, solution of the sup-inf problem \eqref{def:infsuprigorous} obtained in the previous Theorem, as being a {\bf variational dual solution} of Euler. A similar construct for nonlinear elastostatics is obtained in \cite{sga}.
\end{remark} 

\begin{remark} We would like  to emphasize that the base state $\bar V$ needs only to be in $L^2_T(L^2(\TTd))$ and in particular it need not be divergence free.
\end{remark}

\begin{remark}
One issue that remains unclear is what is the relationship of the variational dual solutions with the weak solutions of Euler. We will see below that weak solutions define variational dual solutions, for a suitable base state. However it is not clear if the reciprocal also holds.

Indeed, if $(\lambda^*,\gamma^*)\in\mathfrak{R}$ provide a variational dual solution, then  formally, according to the scheme  in Section~\ref{sec:Eulerheuristic}, we have that $V^{i,\mcH}:=(N^{-1})^{ij}(a_V)^{-1}\left(-\partial_j\gamma-\partial_t \lambda_j+a_V \bV^j \right)$ is a weak solution. However out of $(\lambda^*,\gamma^*)\in\mathfrak{R}$ we only know that $\partial_j\gamma+\partial_t \lambda_j-a_V \bV_j$ is in $L^2_T(L^2)$ but we do not know if $N(a_V,B)^{-1}$ is bounded in $L^\infty$ (nor in any $L^p$ for that matter). Obtaining such an upper bound on $N(a_V,B)^{-1}$ seems to be necessary in order to understand in what space lie the solution of Euler obtained out of the variational dual solutions.  
\end{remark}

\begin{remark}
Following Brenier, we could have attempted to do the theorem in terms of the variables $E,B$ which are natural for this setting, in which case these would belong $\mathcal{EB}_{2,\infty}$ a subspace of $L^2\times L^\infty$ functions such that 

\begin{equation}\label{rel:testcompatibilityEB2infty}
    \partial_t B_{ij}=\frac{1}{2}(\partial_j E_i+\partial_i E_j)-\partial_i\partial_j\Delta^{-1}\partial^l E_l
\end{equation} holds in a weak sense, see \cite{Brenier-CMP} for details. However, this would have worked only if the forcing term $F$ were zero, since this generates a term in the dual formulation that cannot be expressed in terms of $E$ and $B$.
\end{remark}

 \begin{remark}
Overall, while the goal is to obtain a weak solution of the primal problem, noting that the pre-dual $\widehat{S}$ is necessarily affine (and hence concave) in the dual fields, the $\inf$ in the statement of the previous theorem renders the resulting functional concave in the dual variables. Thus, solving a $\sup_D \inf_U \widehat{S}[U,D;\bar{U}]$ problem is a `relaxed' version of seeking a critical point with a consistency guarantee that were a maximizer of $\inf_U \widehat{S}$ to exist and be within the region where the functional is differentiable, one would have obtained a weak solution of the primal problem, using the Dual-to-Primal map ($V^H(\mathcal{D})$/\eqref{eq:VHsimple}), through a critical point of the dual variational problem.

\end{remark}

Let us now prove the consistency of our construction.

\begin{theorem}\label{thm:consistency}

Let $\tilde V\in L^2_T(L^2(\TTd))$ be a weak solution of the Euler with initial data $V^0\in L^2(\TTd)$, mean-zero and with $\nabla\cdot V^0\equiv 0$. Then  the sup-inf problem \eqref{def:infsuprigorous} with base state $\bV=\tilde V$ has $(\lambda, \gamma)=(0,0)$  as a solution, hence it is a variational dual solution to which   $V^\mcH=\tilde V$ corresponds as a solution of Euler. 
\end{theorem}
\begin{proof}

Let us denote

\bea  \widetilde{\mathcal{I}}_E [V,\tilde V](\lambda,\gamma):= &\int_0^T\iTTd \left(V^i\partial_t \lambda_i+\frac{1}{2}V^iV^j(\partial_j\lambda_i+\partial_i\lambda_j) \right)\,dtdx+\iTTd \lambda_i(x,0)V^i_0(x),dx\non\\
&+\int_0^T\iTTd
V^i(t,x)\partial_i\gamma(t,x)\,dtdx+\int_0^T\iTTd \frac{a_V}{2} |V(t,x)- \tilde V(t,x)|^2\,dtdx
\eea
with $(\lambda, \gamma)\in \mathfrak{R}$
and, furthermore define:

\be
\mathcal{I}_{E}[\tilde{V}](V_0):=\inf_V \sup_{(\lambda,\gamma)\in\mathfrak{R}}  \widetilde{\mathcal{I}}_E[V,\tilde{V}](\lambda,\gamma)
\ee

Similarly as in the work of Brenier \cite{Brenier-CMP} we note that since $\inf\sup\ge \sup\inf$ we have:

\be
\mathcal{I}_{E}[\tilde{V}](V_0)\ge \mcJ_{E}[\tilde{V}](V_0)
\ee .

On the other hand, out of the definition of $\mathcal{I}_{E}[\tilde{V}](V_0)$ we have:

\be
\mathcal{I}_{E}[\tilde{V}](V_0)\le \sup_{(\lambda,\gamma)\in\mathfrak{R}}\widetilde{\mathcal{I}}[\tilde V,\tilde{V}](\lambda,\gamma)=0
\ee where the last equality holds because of our assumption that $\tilde V$ is a weak solution of Euler.

Out of the last two inequalities we have that:

\be
0\ge \mcJ_{E}[\tilde{V}](V_0)
\ee

On the other hand, taking $(\lambda,\gamma)=(0,0)$ we have:

\be
\mcJ_{E}[\tilde{V}](V_0)\ge \inf_V \int_0^T\iTTd \frac{a_V}{2}|V-\tilde V|^2\,dtdx\ge 0
\ee

Thus the last two relations show that the value $\mcJ_{E}[\tilde{V}](V_0)=0$ is attained for $(\lambda,\gamma)=0$ and $V=\tilde V$. 

\end{proof}

\begin{remark}
It should be noted that we are able to show that any weak solution of Euler is obtained from a variational dual solution of Euler. This is in contrast to \cite{Brenier-CMP}, Thm. $2.3$, where the consistency of the maximization scheme is shown only for strong solutions of Euler, local in time, and for which certain quantitative assumptions are satisfied. 
\end{remark}

\section{Variational Dual Solutions for Navier-Stokes}

\label{sec:NS}

We proceed analogously to the case of Euler and aim to obtain solutions through an sup-inf problem, that now becomes:

\bea
\sup_{(\lambda,\gamma, \chi)}\inf_{V,W}&\int_0^T\iTTd\left(V^i\partial_i\gamma+V^i\partial_t \lambda_i+\frac{1}{2}V^iV^j(\partial_j\lambda_i+\partial_i\lambda_j)- W^{ij}\partial_j \lambda_i \right)\,dtdx\non\\
&+\int_0^T\iTTd\left( 2\nu V^i\partial_j \chi_{ij}+W^{ij}\chi_{ij}\right)\,dtdx+\non\\
&+ \int_0^T\iTTd F^i(t,x)\lambda_i (t,x)\,dtdx+\iTTd  \lambda_i(x,0)V^i_0(x),dx\non\\
&+\int_0^T\iTTd\fH(V, W;\bV,\bW))\,dtdx
\eea

Similarly as in the case of Euler, we observe that  in order to have  the infimum in $V$ exists,  we need to have pointwise  $a_V\Id+2B\ge 0$ ( where the matrix $B_{ij}:=\frac{1}{2}(\partial_i\lambda_j+\partial_j\lambda_i)$) which will lead, with the same arguments as in Section~\ref{sec:Euler} to the bound

\begin{equation}\label{lwbd:B}
-a_V\mathbb{I}_d \le 2B\le a_V(d-1)\mathbb{I}_d,
\end{equation}
which provides an  apriori $L^\infty$ bound for $B$ (since $B$ is valued in the set of symmetric
matrices). 

This time we will need the variable:
 $$
 \mbE_i:=\partial_t \lambda_i+\partial_i \gamma+2\nu \partial_l\chi_{il}
 $$ 

Similarly to the case of Euler, the variables $\mbE$ and $B$ will be the natural variables in which to obtain existence of the sup-inf problem. However, unlike in the case of Euler, we will no longer be able to obtain separate regularity of $\lambda,\gamma$ out of the regularity of $\mbE$ and $B$, because of the presence of the term $\partial_l\chi_{il}$ for which the sup-inf problem will only provide $H^{-1}$ regularity. Thus we have to adopt a functional framework similar to the one used by Brenier for Euler in \cite{Brenier-CMP}. Also, because of this issue, we will need to take $F\equiv 0$.

Note that we have the identities:

\begin{equation}
\partial_t B_{ij}=\frac{1}{2}(\partial_j \mbE_i+\partial_i \mbE_j)-\partial_i\partial_j\gamma-\nu (\partial_j\partial_l\chi_{li}+\partial_i\partial_l\chi_{lj})
\end{equation} and

\begin{equation}
\partial_i \mbE_i=\Delta\gamma+2\nu\partial_i\partial_l \chi_{li}
\end{equation} hence

\begin{equation}
\gamma:=\Delta^{-1}(\partial_i \mbE_i-2\nu\partial_i\partial_j\chi_{ij})
\end{equation} which provides a compatibility relationship between the new unknowns $\mbE,B,\chi$ namely:

\begin{equation}\label{rel:testcompatibility}
    \partial_t B_{ij}=\frac{1}{2}(\partial_j \mbE_i+\partial_i \mbE_j)-\partial_i\partial_j\Delta^{-1}(\partial_l \mbE_l-2\nu\partial_l\partial_k\chi_{lk})-\nu(\partial_j\partial_l\chi_{li}+\partial_i\partial_l\chi_{lj})
\end{equation}

Let $\mathcal{EB}\Psi_{2,\infty,2}^\nu$ be the class of all $L^2\times L^\infty\times L^2$ fields $(E,B,\chi)$ on
$[0,T]\times \TTd$ with $E$
 valued
in $\mathbb{R}^d$ and $B$ and $\chi$ taking values   into the set of 
symmetric $d\times d$ matrices satisfying weakly the constraint \eqref{rel:testcompatibility} with $B(T,\cdot)\equiv 0$. 

We then have the following analogue of Brenier's Theorem $2.2$ namely:

\begin{theorem} 
\label{existenceNS}
Let $V_0\in L^2(\Omega,\mathbb{R}^d)$ be 
a divergence-free vector field 
of zero spatial mean with $\nabla V_0\in L^2(\Omega,\mathbb{R}^{d\times d}) $. 
Then the sup-inf problem 
\bea
\mcJ_{NS}[\bV,\bW](V_0):=&\sup_{(E,B,\chi)\in \mcE\mcB\Psi_{2,\infty,2}^\nu}\inf_{V,W} \int_0^T\iTTd V^i\mbE_i+\frac{1}{2} V^iV^j (2B_{ij}+a_V\delta_{ij})\non\\
&+\int_0^T\iTTd\frac{a_V}{2} \bV^i(\bV^i-2V^i)\,dtdx\non\\
&+\int_0^T\iTTd \frac{a_W}{2}W^{ij}W^{ij}-{W^{ij}(B_{ij}-\chi_{ij})}\,dtdx\non\\
&+\int_0^T\iTTd\frac{a_W}{2}\bW^{ij}\bW^{ij}-{a_W \bW^{ij} W^{ij}}\,dtdx\non\\
&-\int_0^T\iTTd V^i_0(x)\mbE_i(t,x){+}2\nu\partial_l V^i_0 \chi_{il}\,dtdx
\label{infsupNS}
\eea  admits a solution $(E,B,\chi)\in \mcE\mcB\Psi_{2,\infty,2}^\nu$.
In addition $B$ belongs to $C^{1/2}([0,T],C^2(\TTd)'_{w*})$.

\end{theorem}

\begin{proof}
Minimization in $V$ and $W$ provides
\bea\label{ivpdual0}
\mcJ_{NS}[\bV,\bW](V_0)=&\sup_{(\mbE,B,\chi)\in\mathcal{EB}\Psi_{2,\infty,2}^\nu}
\;\;-\tilde{\mcM}[\mbE-a_V \bV,B,\chi]-\int_0^T\iTTd
V^i_0 \mbE_i\,+2\nu\partial_l V_i^0\chi_{il}\,dtdx
\eea
where 

\bea
\tilde{\mcM}[\mbE,B, \chi]\colon=\mcM[\mbE,B]+\int_0^T\iTTd \frac{1}{2a_W}(\chi_{ij}-B_{ij}-a_{W}\bW^{ij})( \chi_{ij}-B_{ij}-a_{W}\bW^{ij})\,dtdx
    \eea with
\begin{equation}\label{M00}
\mcM[\mbE,B]=\frac{1}{2}\int_0^T\iTTd \mbE\cdot N(a_V,B)^{-1}\cdot \mbE\,dtdx
\end{equation} and

 \be
 N(a_V,B):=a_V\mathbb{I}_d+2B
 \ee 

This immediately implies $\mcJ_{NS}[\bV,\bW](V_0)\ge 0$
(just by taking $\mbE=B=\chi=0$).
\\
We note that we  can write, point-wise in $(t,x)$,
$$
\mbE\cdot(\mathbb{I}_d+2B)^{-1}\cdot \mbE
=\sup_{M,Z} \;\;2\mbE_iZ^i-(\delta_{ij}+2B_{ij})M^{ij}
$$
where $M$ and $Z$ are respectively $d\times d$ symmetric
matrices and vectors in $\mathbb{R}^d$ subject to
\begin{equation}\label{matrineq}
Z\otimes Z\le M,
\end{equation}
in the sense of symmetric matrices. This allows us to give an alternative definition of $\mcM$, namely
\begin{equation}\label{M}
\mcM[\mbE,B]=\sup_{M\ge Z\otimes Z} \;\;\frac{1}{2}\int_0^T\iTTd
2 \mbE_iZ^i-(\delta_{ij}+2B_{ij})M^{ij}
\in [-\infty,0],
\end{equation}
where the supremum is performed over all pairs $(Z,M)$ 
of continuous functions on
$[0,T]\times \TTd$, respectively valued
in $\mathbb{R}^d$ and in the set of 
symmetric matrices $M$.
Notice that definition (\ref{M00}) makes sense as $(\mbE,B)$ belong to $L^2\times L^\infty$. Moreover, $\mcM$ is a lower semi-continuous  function of $(\mbE,B)$ valued in $[0,+\infty]$.
Next, because of the lower bound \eqref{lwbd:B}, we get an $L^2$ apriori bound for $\mbE-a_V\bV$. Indeed,
by definition (\ref{M}) of $\mcM$, we have:
\bea
\frac{1}{2a_Vd}\int_0^T\iTTd|\mbE-a_V\bV|^2&\le \mcM[\mbE-a_V\bV,B]\non\\
&=\mcM[\mbE,B]
-\frac{1}{2}\int_0^T\iTTd
a_V\bV\cdot N(a_V,B)^{-1}\cdot\mbE\,dtdx\non\\
&-\frac{1}{2}\int_0^T\iTTd a_V\mbE\cdot N(a_V,B)^{-1}\cdot \bV\,dtdx\non\\
&+\frac{1}{2}\int_0^T\iTTd a_V\bV\cdot N(a_V,B)^{-1} \cdot a_V\bV\,dtdx\non
\eea
So, for any $\varepsilon-$maximizer $(\mbE,B,\chi)\in \mathcal{EB}\Psi_{2,\infty,2}^\nu$ of (\ref{ivpdual0}), we get (where $\|\cdot\|$ denotes the $L^2_T(L^2)$ norm, for simplicity)

\bea
&0\le \mcJ_{NS}[\bV,\bW](V_0)
\le \varepsilon-\mcM[\mbE-a_V\bV,B]-\int_0^T\iTTd V_0\cdot \mbE+2 \nu \partial_j V^i_0\chi_{ji}\,dtdx,\non\\
&-\frac{1}{2 a_W}\|\chi-B-a_W\bW\|^2+\frac{a_V}{2}\|\bV\|^2+\frac{a_W}{2}\|\bW\|^2\non\\
&\le \varepsilon-\frac{1}{4da_V}\int_0^T\iTTd(2|\mbE-a_V\bV|^2+4da_V V_0\cdot \mbE)\non\\
&-\frac{1}{2a_W}\|\chi-B-a_W\bW\|_{L^2}^2+\frac{a_V}{2}\|\bV\|^2+\frac{a_W}{2}\|\bW\|^2+\int_0^T\iTTd 2\nu\partial_j V_0^i\chi_{ji}\,dtdx,\non
\eea
\bea
&= \varepsilon-\frac{1}{4da_V}\Bigg(2\|\mbE\|^2+2\|\frac{\mbE}{2}-2a_V\bV\|^2-\frac{\|\mbE\|^2}{2}-6a_V^2\|\bV\|^2+\|4da_VV_0+\frac{\mbE}{2}\|^2-\frac{\|\mbE\|^2}{4}-16d^2a_V^2\|V_0\|^2\Bigg)\non\\
&-\frac{1}{2a_W}\Big(\|\chi\|^2+\|2(B+a_W\bW)-\frac{\chi}{2}\|^2-3\|B+a_W\bW\|^2-\frac{\|\chi\|^2}{4}\Big)+\frac{a_V}{2}\|\bV\|^2+\frac{a_W}{2}\|\bW\|^2\non\\
&-\|2\nu\sqrt{2a_W}\nabla V_0+\frac{\chi}{2\sqrt{2a_W}}\|^2+8a_W\nu^2\|\nabla V^0\|^2+\frac{1}{2}\frac{\|\chi\|^2}{4 a_W}\non\\
&\le \varepsilon-\frac{1}{4da_V}\Bigg(\frac{5}{4}\|\mbE\|^2-6a_V^2\|\bV\|^2-16d^2a_V^2\|V_0\|^2\Bigg)-\frac{1}{2 a_W}\Big(\frac{\|\chi\|^2}{2}-3\|B+a_W\bW\|^2\Big)\non\\
&+8a_W\nu^2\|\nabla V^0\|^2+\frac{a_V}{2}\|\bV\|^2+\frac{a_W}{2}\|\bW\|^2\label{estJNSVW}
\eea

which provides the a priori bound, for every $\varepsilon-$maximizer $(\mbE,B,\chi)\in \mathcal{EB}\Psi_{2,\infty,2}^\nu$ of (\ref{ivpdual0}), namely:
\bea\label{bound-Echi}
\frac{5\|\mbE\|^2}{16da_V}+\frac{\|\chi\|^2}{4 a_W}\le &\eps+\frac{6a_V\|\bV\|^2}{4d}+4da_V\|V_0\|^2+\frac{3}{2 a_W}\|B+a_W\bW\|^2+{ 8  a_W}\nu^2\|\nabla V^0\|^2
\non\\
&+\frac{a_V}{2}\|\bV\|^2+\frac{a_W}{2}\|\bW\|^2\non\\
\le & \eps+\frac{a_V(3+d)}{2d}\|\bV\|^2+\frac{7a_W}{2}\|\bW\|^2+4da_V\|V_0\|^2+{ 8  a_W}\nu^2\|\nabla V^0\|^2+\frac{3}{2 a_W}\|B\|^2\non\\
\le  &\eps+\frac{a_V(3+d)}{2d}\|\bV\|^2+\frac{7a_W}{2}\|\bW\|^2+4da_V\|V_0\|^2+{ 8  a_W}\nu^2\|\nabla V^0\|^2\non\\
&+\frac{3a_V^2}{4 a_W}(d-1)^2|\mathbb{T}^d| T
\eea where for the last inequality we used the estimate \eqref{lwbd:B}.

Out of relation \eqref{estJNSVW} we get, by dropping the negative terms on the right hand side of the last inequality and letting $\eps\to 0$:

\bea
\mcJ_{NS}[\bV,\bW](V_0)\le &\frac{1}{4da_V}\Bigg(6a_V^2\|\bV\|^2+16d^2a_V^2\|V_0\|^2\Bigg)+\frac{3}{2}\|B+a_W\bW\|^2\non\\
&+{ 8  a_W}\nu^2\|\nabla V^0\|^2+\frac{a_V}{2}\|\bV\|^2+\frac{a_W}{2}\|\bW\|^2\non\\
\le  &\frac{a_V(3+d)}{2d}\|\bV\|^2+\frac{7a_W}{2}\|\bW\|^2+4da_V\|V_0\|^2+{ 8  a_W}\nu^2\|\nabla V^0\|^2\non\\
&+\frac{3a_V^2}{4 a_W}(d-1)^2|\mathbb{T}^d| T
\eea
We notice that  the upper bound is expressed in terms of the base states and the initial datum. 

By definition (\ref{M}), $(\mbE+\bV,B)\to \mcM[\mbE+\bV,B]$ is lower semi-continuous with respect to the weak-* topology of 
$L^\infty \times L^2\times L^2$ and the quadratic term in $\chi,B$ in $\tilde\mcM$ is also lower semi-continuous with respect to the weak topology in $L^2$. Also
$$
\mbE\rightarrow \int_0^T\int_\Omega V_0\cdot \mbE,
$$ and 

$$
\chi\to\int_0^T\int_\Omega 2\nu\partial_j V^i_0\chi_{ij}\,dx
$$
are continuous (since $V_0,\nabla V_0$ is given in $L^2$) and $\mathcal{EB}\Psi_{2,\infty,2}^\nu$ is weak-*
closed in $L^2\times L^\infty \times L^2$.
Thus, we conclude that the maximization problem
(\ref{ivpdual0}) 
always has at least one optimal solution $(\mbE,B,\chi)$ in class $\mathcal{EB}\Psi_{2,\infty,2}^\nu$, since its 
$\varepsilon-$maximizers stay confined in a fixed ball (and therefore a 
weak-* compact subset) 
of $L^2\times L^\infty \times L^2$, as $\varepsilon$ goes to zero. 

As for the regularity of $B$, we can argue similarly as in \cite{Brenier-CMP}. Indeed, we can write the righ-hand side of equation \eqref{rel:testcompatibility} as a sum of a first-order operator in $\mbE$ and a remainder:
\begin{equation}\label{rel:compatibility}
    \partial_t B_{ij}=(L\mbE)_{ij}+2\nu\partial_i\partial_j\Delta^{-1}(\partial_l\partial_k\chi_{lk})-\nu(\partial_j\partial_l\chi_{li}+\partial_i\partial_l\chi_{lj})
\end{equation}
where
\bea\label{operatorL}
(L\mbE)_{ij}=\frac{1}{2}(\partial_j \mbE_i+\partial_i \mbE_j)-\partial_i\partial_j\Delta^{-1}(\partial_l \mbE_l)
\eea

Then, for any smooth function $\psi$ on $\TTd$ valued on symmetric $d\times d$ matrices, we have
\bea
\iTTd (B_{ij}(t_1, \cdot)-B_{ij}(t_0, \cdot))\psi^{ij}\le \sqrt{t_1-t_0}\Big[\|\mbE\|_{L^2([0,T]\times \TTd)}\|L^*\psi\|_{L^{\infty}(\TTd)}+\|D^2\psi\|_{L^{\infty}(\TTd)}\|2\nu \chi\|_{L^2}\Big]
\eea where the operator $L^*$ is defined as 
$(L^*)^k_{ij}\psi^{ij}
=\partial_j\psi^{kj}-\partial^k(-\bigtriangleup)^{-1}\partial_i\partial_j\psi^{ij}.$

Finally, using the estimate \eqref{bound-Echi}, we can bound from above with terms depending on the base states and the initial datum.
\end{proof}

\begin{remark}
Let us note that with the current argument we are unable to specifiy the functional space in which the $V^\mcH,W^\mcW$ are. If one would be able to provide an $L^\infty$ bound of $N(a_V,B)^{-1}$ then we would have that the $V^\mcH,W^\mcW$ are in $L^2_T(L^2)$.
\end{remark}

Similarly as for Euler, we can also prove the consistency of our construction for Navier-Stokes, namely:

\begin{theorem}

Let $\tilde V\in L^2_T(H^1(\TTd))$ be a strong solution of the Navier-Stokes with initial data $V^0\in H^1(\TTd)$, mean-zero and with $\nabla\cdot V^0\equiv 0$.Then $\tilde W^{ij}:=\nu (\partial_j \tilde V^i+\partial_i \tilde V^j)$ belongs to $ L^2_T(L^2(\TTd)) $ and  the sup-inf problem \eqref{infsupNS}
with base state $(\bV,\bW)=(\tilde V,\tilde W)$ has $(\lambda,\gamma,\chi)=(0,0,0)$ as a solution, hence it is a variational dual solution to which   $(V^\fH,W^\fH)=(\tilde V,\tilde W)$ corresponds as a solution of Navier-Stokes.

\end{theorem}
\begin{proof}

Let us denote

\bea  \widetilde{\mathcal{I}}_{NS}[V,W;\tilde V,\tilde W](\mbE,B,\chi):= &\int_0^T\iTTd V^i\mbE_i+\frac{1}{2} V^iV^j (2B_{ij}+a_V\delta_{ij})\non\\
&+\int_0^T\iTTd \frac{a_V}{2} \bV^i(\bV^i-2V^i)\,dtdx\non\\
&+\int_0^T\iTTd \frac{a_W}{2}W^{ij}W^{ij}+W^{ij}(\chi_{ij}- B_{ij})\,dtdx\non\\
&+\int_0^T\iTTd\frac{a_W}{2}\bW^{ij}(\bW^{ij}-2W^{ij})\,dtdx\non\\
&-\int_0^T\iTTd (V^i_0(x)\mbE_i(t,x)+2\nu\partial_l V^i_0 \chi_{il})\,dtdx
\eea
with $(\mbE, B,\chi )\in \mathcal{EB}\Psi_{2,\infty,2}^\nu$
and, furthermore define:

\be
\mathcal{I}_{NS}[\tilde{V},\tilde W](V_0):=\inf_{V,W} \sup_{(\mbE, B,\chi )\in \mathcal{EB}\Psi_{2,\infty,2}^\nu}  \widetilde{\mathcal{I}}_{NS}[V,W;\tilde V,\tilde W](\mbE,B,\chi)
\ee

Similarly as in the work of Brenier \cite{Brenier-CMP} we note that since $\inf\sup\ge \sup\inf$ we have:

\be
\mathcal{I}_{NS}[\tilde{V},\tilde W](V_0)\ge \mcJ_{NS}[\tilde{V},\tilde W](V_0)
\ee .

On the other hand, out of the definition of $\mathcal{I}_{NS}[\tilde{V},\tilde{W}](V_0)$ we have:

\be
\mathcal{I}_{NS}[\tilde{V},\tilde W](V_0)\le \sup_{(\mbE, B,\chi )\in \mathcal{EB}\Psi_{2,\infty,2}^\nu}  \widetilde{\mathcal{I}}_{NS}[\tilde V,\tilde W;\tilde V,\tilde W](\mbE,B,\chi)=0
\ee where the last equality holds because of our assumption that $\tilde V$ is a weak solution of Navier-Stokes and $\tilde W^{ij}=\nu (\partial_j \tilde V^i+\partial_i \tilde V^j)$.

Out of the last two inequalities we have that:

\be
0\ge \mcJ_{NS}[\tilde{V},\tilde W](V_0)
\ee

On the other hand, taking $(\mbE,B,\chi)=(0,0,0)$ we have:

\be
\mcJ_{NS}[\tilde{V},\tilde W](V_0)\ge \inf_V \int_0^T\iTTd \frac{a_V}{2}|V-\tilde V|^2+\frac{a_W}{2}|W-\tilde W|^2\,dtdx\ge 0
\ee

Thus the last two relations show that the value $\mcJ_{NS}[\tilde{V},\tilde W](V_0)=0$ is attained for $(\mbE,B,\chi)=(0,0,0)$ and $V^i=\tilde V^i$, $W^{ij}=\tilde W^{ij}=\nu (\partial_j \tilde V^i+\partial_i \tilde V^ j)$. 

\end{proof}

\section{The inviscid limit of Navier-Stokes to Euler as a $\Gamma$-convergence  }

In this section we will show that using the framework of variational dual solutions as defined before we can obtain the convergence of solutions for Navier-Stokes to solutions of Euler, when $\nu\to 0$.

To this end we will show the $\Gamma$ convergence of a sequence of functionals intimately involved in studying the mentioned limit. 
Thus let us consider the functionals  (for $\nu\ge 0$) :

\bea\label{def:LNS}
\mathcal{A}^\nu_{NS}(\mbE,B,\chi):=&\mcM[\mbE-(a_V+\nu^\alpha)\bV,N(a_V+\nu^\alpha)]+\int_0^T\iTTd
V^i_0 \mbE_i\,+2\nu\partial_l V^i_0\chi_{il}\,dtdx\non\\
&+\int_0^T\iTTd \frac{1}{2}(\chi_{ij}-B_{ij}-a_W\bW^{ij})(\chi_{ij}- B_{ij}-a_W\bW^{ij})\,dtdx\non\\
&-\int_0^T\int_{\TTd} \frac{a_V+\nu^\alpha}{2}\bV^i\bV^i+\frac{a_W}{2}\bW^{ij}\bW^{ij}\,dtdx
\eea  where  we denoted

\begin{equation}\label{def:mcM}
\mcM[\mbE,N]:=\frac{1}{2}\int_0^T\iTTd \mbE_i (N^{-1})^{ij}\mbE_j\,dtdx
\end{equation} with 

$$
\mbE_i, B_{ij}, N_{ij}(a):=a \delta_{ij}+2B_{ij};\, i,j=1,\dots,d
$$ as previously defined in the Section~\ref{sec:NS} on Navier-Stokes. 

We also take the exponent $\alpha$ to be such that 

\be\label{ass:alpha}
\alpha<\frac{1}{[d/2]+4}.
\ee

For any $\nu\ge 0$   we denote as before $\mathcal{EB}\Psi_{2,\infty,2}^\nu$ be the class of all $L^2\times L^\infty\times L^2$ fields $(\mbE,B,\chi)$ on
$[0,T]\times \TTd$ with $\mbE$
 valued
in $\mathbb{R}^d$ and $B$ and $\chi$ taking values   into the set of 
symmetric $d\times d$ matrices satisfying weakly the constraint \eqref{rel:testcompatibility}  with $B(T,\cdot)\equiv 0$.  

We consider the metric space $$\widetilde{\mathcal{EB}\Psi_{2,\infty,2}}:=\cup_{\nu\ge 0}\mathcal{EB}\Psi_{2,\infty,2}^\nu$$ endowed with the metric induced by the ambient space $L^2\times L^\infty\times L^2$ namely $$d((\mbE_1,B_1,\chi_1),(\mbE_2,B_2,\chi_2))=\|\mbE_1-\mbE_2\|_{L^2_T(L^2)}+\|B_1-B_2\|_{L^\infty_T(L^\infty))}+\|\chi_1-\chi_2\|_{L^2_T(L^2)}.$$

For $\nu\ge 0$ we define:

\be
\widetilde{\mathcal{A}^\nu_{NS}}(\mbE,B,\chi)=\left\{\begin{array}{ll}
\mathcal{A}^\nu_{NS}(\mbE,B,\chi) &\textrm{ for }(\mbE,B,\chi)\in \mathcal{EB}\Psi_{2,\infty,2}^\nu\\
+\infty &\textrm{ for }(\mbE,B,\chi)\in \widetilde{\mathcal{EB}\Psi_{2,\infty,2}}\setminus \mathcal{EB}\Psi_{2,\infty,2}^\nu
\end{array}\right.
\ee

Let us note that for $\nu=0$ we can naturally identify  $\mathcal{EB}\Psi_{2,\infty,2}^0 $ with $\mathcal{EB}_{2,\infty}\times L^2_T(L^2)$ as $\chi$ is no longer constrained through relation \eqref{rel:testcompatibility} if $\nu=0$.

Moreover, if we denote:

\bea
\mathcal{A}_{E}(E,B):=& \mcM [\mbE-a_V\bV,N(a_V,B)]+\int_0^T\iTTd
V^i_0 \mbE_i\,dtdx\non\\
&-\int_0^T\iTTd \frac{a_V}{2}\bV^i\bV^i+\frac{a_W}{2}
\bW^{ij}\bW^{ij}\,dtdx
\eea we see that the minimizers of this provide precisely the variational dual solutions of the Euler problem with initial data $V_0$ (as the additive constant $-\int_0^T \frac{a_W}{2}\bW^{ij}\bW^{ij}$ does not affect the minisation). 
Furthermore, we have that for any $(\mbE,B,\chi)\in \mathcal{EB}\Psi_{2,\infty,2}^0$ we have:

\be
\mathcal{A}^0_{NS}(\mbE,B,\chi) \ge \mathcal{A}_{E}(\mbE,B)
\ee so the minimizers of $\mathcal{A}^0_{NS}$ are provided by minimizers $(\mbE^*,B^*)$ of $\mathcal{A}_{E}$ and $\chi^*:=B^*+a_W\bW$.

We then have the following:

\begin{proposition}\label{prop:gammaconvstrong}
Let $V_0\in L^2(\Omega,\mathbb{R}^d)$ be 
a divergence-free vector field 
of zero spatial mean with $\nabla V_0\in L^2(\Omega,\mathbb{R}^{d\times d}) $. Then 

\be
\widetilde{\mathcal{A}^\nu_{NS}}\stackrel{\Gamma}{\longrightarrow} \widetilde{\mathcal{A}^0_{NS}}
\ee  in $\widetilde{\mathcal{EB}\Psi_{2,\infty,2}}$.
\end{proposition}
\begin{proof}
We first consider the liminf limit. We notice that we can omit the continuous linear terms. Similarly the quadratic terms in \eqref{def:LNS} are weakly lower semicontinuous  with respect to the weak convergence in $\mathcal{EB}\Psi_{2,\infty,2}$.  We focus only on the term $\mcM$. We recall that we have:
\begin{align}
&\mcM[\mbE-(a_V+\nu^\alpha)\bV,N(a_V+\nu^\alpha)]=\nonumber\\
&\sup_{M\ge Z\otimes Z} \;\;\frac{1}{2}\int_0^T\iTTd
2 (\mbE_i-(a_V+\nu^\alpha)\bV^i) Z^i-N_{ij}(a_V+\nu^\alpha)M^{ij}
\in [-\infty,0],
\label{K}
\end{align}
where the supremum is performed over all pairs $(Z,M)$ 
of continuous functions on
$[0,T]\times \TTd$, respectively valued
in $\mathbb{R}^d$ and in the set of 
symmetric matrices $M$. Thus $\mcM$ being a pointwise supremum of affine functions is a convex function and also it is lower semicontinuous with respect to the weak convergence in $\mathcal{EB}\Psi_{2,\infty,2}$, so we get the liminf inequality.

For the limsup, we take  $(\bar{\mbE},\bar{B})\in L^2_T(L^2)\times L^\infty_T(L^\infty)$  so that, by denoting the first order operator

$$
(L\mbE)_{ij}:=\frac{1}{2}(\partial_j \mbE_i+\partial_i \mbE_j)-\partial_i\partial_j\Delta^{-1}\partial_l \mbE_l
$$

we have, in a weak sense:

\be\label{constr:EBEuler}
\partial_t \bar{B}=L\bar{\mbE}
\ee

We take an arbitrary function $\bar \chi\in L^2_T(L^2)$. For the limsup inequality we aim to find $(\mbE^\nu,B^\nu,\chi^\nu)\in\mathcal{EB}\Psi_{2,\infty,2}^\nu $ such that $d((\mbE^\nu,B^\nu,\chi^\nu),(\bar{\mbE},\bar{B},\bar\chi))\to 0$ as $\nu\to 0$ and also that 
\bea\label{rel:limsup}
\limsup_{\nu}\widetilde{\mathcal{A}^\nu_{NS}}(\mbE^\nu,B^\nu,\chi^\nu)\le 
\widetilde{\mathcal{A}^0_{NS}}(\bar{\mbE},\bar{B},\bar{\chi})
\eea

To this end we take $\chi^\nu\in C^\infty_c(\R\times \R^d)$ to be such that

\be
\chi^\nu\to\bar\chi\textrm{ in }L^2_T(L^2)\textrm{ and a.e.}
\ee and

\be \label{rel:Highderivchi}
\nu\|\chi^\nu\|_{L^2_T(H^{[d/2]+3})}\le C\nu^{\frac{1}{[d/2]+4}}\to 0
\ee  hence in particular, by the Sobolev embeddings:

\be\label{conv:nudeltachi}
|\nu\partial_i\partial_j \chi_{kl}^\nu|_{L^\infty}\le \tilde{C}\nu^{\frac{1}{[d/2]+4}}\to  0, \forall i,j,k,l=1,\dots,d
\ee where the constants $C$ and $\tilde C$ in the last two relations do not depend on $\nu$.

We define 

$$\mbE^\nu:=\bar{\mbE}$$

and then take

\be\label{def:Bnu}
B^\nu:=\bar{B}-2\nu\int_t^T\partial_i\partial_j\Delta^{-1}(\partial_k\partial_l\chi_{kl}^\nu)-\frac{1}{2}(\partial_j\partial_l\chi_{li}^\nu+\partial_i\partial_l\chi_{lj}^\nu)\,ds
\ee

Then  we have

$$
\partial_t B^\nu_{ij}=(L\mbE^\nu)_{ij}+2\nu\partial_i\partial_j\Delta^{-1}(\partial_i\partial_j\chi_{ij}^\nu)-\nu(\partial_j\partial_l\chi_{li}^\nu+\partial_i\partial_l\chi_{lj}^\nu)
$$ hence $E^\nu,B^\nu,\chi^\nu$ satisfy the compatibility relation to be in the space $\mathcal{EB}\Psi_{2,\infty,2}^\nu$.

In order to obtain the relation   \eqref{rel:limsup} we note that we can pass to the limit in the second integral in \eqref{def:LNS}  and also in the  integral containing $V^0 $ thanks to the strong convergence of $\chi^\nu$ and $B^\nu$ in $L^2$ and \eqref{rel:Highderivchi}. Moreover, for $\nu$ small enough we have (thanks to the definition \eqref{def:Bnu} of $B^\nu$, the bound \eqref{conv:nudeltachi} and the assumption \eqref{ass:alpha} on $\alpha$):

\be\label{rel:avnavn}
(a_V+\nu^\alpha)Id+2B^\nu\ge  Id+2\bar B\ge 0
\ee 

Furthermore, for $\nu$ small enough, we also have:

\be\label{rel:avnnu}
(1+\nu^\alpha)Id+2B^\nu\ge (a_V+\frac{\nu^\alpha}{2})Id+2\bar{B}\ge \frac{\nu^\alpha}{2}Id
\ee

We take now $R^\nu:[0,T]\times\TTd\to \mathbb{O}(d)$ (where $\mathbb{O}(d)$ denotes the group of orthogonal matrices) to be such that at almost all $(t,x)\in [0,T]\times\TTd$ we have

\be
R^\nu((a_V+\nu^\alpha)Id+2B^\nu)^{-1}(R^\nu)^t=D^\nu
\ee where we denoted:

\be
D^\nu:=\textrm{diag}(f_1^\nu,\dots, f_d^\nu)
\ee with $f_i^\nu\ge 0, i=1,\dots,d $  the eigenvalues of $\bigg((a_V+\nu^\alpha)Id+2B^\nu\bigg)^{-1}$. 

We then have, for small enough $\nu$:

\bea
&\bigg(\mbE^\nu-(a_V+\nu^\alpha)\bV\bigg)\bigg((a_V+\nu^\alpha)Id+2B^\nu\bigg)^{-1}\bigg(\mbE^\nu-(a_V+\nu^\alpha)\bV\bigg)\non\\
&=R^\nu(\mbE^\nu-(a_V+\nu^\alpha)\bV)D^\nu R^\nu(\mbE^\nu-(a_V+\nu^\alpha)\bV)\non\\
&=R^\nu(\bar{\mbE}-(a_V+\nu^\alpha)\bV)D^\nu R^\nu(\bar{\mbE}-(a_V+\nu^\alpha)\bV)\non\\
&= \bigg(R^\nu(\bar{\mbE}-(a_V+\nu^\alpha)\bV)\bigg)_i^2 f_i^\nu\le 2\bigg(R^\nu(\bar{\mbE}-a_V\bV)\bigg)_i^2 f_i^\nu+2\nu^{2\alpha}\bigg(R^\nu\bV\bigg)_i^2 f_i\non
\eea
\bea
&= 2\bigg(\bar{\mbE}-a_V\bV\bigg)\bigg((a_V+\nu^\alpha)Id+2B^\nu\bigg)^{-1}\bigg(\bar{\mbE}-a_V\bV\bigg)\non\\
&+2\nu^{2\alpha}\bV\bigg((a_V+\nu^\alpha)Id+2B^\nu\bigg)^{-1}\bV\non\\
&\le 2\bigg(\bar\mbE-a_V\bV\bigg)\bigg(a_V Id+2\bar{B}\bigg)^{-1}\bigg(\bar\mbE-a_V\bV\bigg)+4\nu^{\alpha}|\bV|^2\non\\
&\le 2\bigg(\bar\mbE-a_V\bV\bigg)\bigg(a_V Id+2\bar{B}\bigg)^{-1}\bigg(\bar\mbE-a_V\bV\bigg)+4|\bV|^2
\eea where for the penultimate inequality we used \eqref{rel:avnavn} and \eqref{rel:avnnu}.

Thus, since the last two terms in the inequality above are integrable we can use the dominated convergence theorem and the suitable pointwise convergence to get

\be
\lim_{\nu\to 0}\mcM[\mbE^\nu-(a_V+\nu^\alpha)\bV,N(a_V+\nu^\alpha,B^\nu)]\le \mcM[\bar{\mbE}-a_V\bV,N(a_V,\bar B)]
\ee
\end{proof}

\bigskip
We further consider the issue of $\Gamma$-convergence in a weaker topology, in which we can also obtain equi-coerciveness. This will allow us in particular to prove the convergence of the minimizers. 

To this end we note that out of estimate \eqref{bound-Echi} in Theorem~\ref{existenceNS} we have that minimizers $(\mbE,B,\chi)$ of the functional $\mathcal{A}^\nu_{NS}(\mbE,B,\chi)$ satisfy the bound (with $\|\cdot\|$ denoting the $L^2_T(L^2(\TTd))$ norm):

\bea\label{bound-Enuaalpha}
\|\mbE\|^2\le  &\frac{8}{5}(a_V+\nu^\alpha)^2(3+d)\|\bV\|^2+\frac{56}{5}d a_W(a_V+\nu^\alpha)\|\bW\|^2+\frac{64}{5}d^2(a_V+\nu^\alpha)^2{5}\|V_0\|^2+\non\\
&+\frac{128}{5}d(a_V+\nu^\alpha)a_W\nu^2\|\nabla V^0\|^2+\frac{12(a_V+\nu^\alpha)^3}{5a_W}(d-1)^2d^2|\mathbb{T}| T
\eea

\bea\label{bound-chialphanu}
\|\chi\|^2
\le  &\frac{2(3+d)}{d}(a_V+\nu^\alpha)a_W\|\bV\|^2+14a_W^2\|\bW\|^2+16d(a_V+\nu^\alpha)a_W\|V_0\|^2\non\\
&+32a_W^2\nu^2\|\nabla V^0\|^2+3(a_V+\nu^\alpha)^2(d-1)^2d|\mathbb{T}| T
\eea while for $B$ we have the uniform estimate  (see beginning of Section~\ref{sec:Euler})

\be\label{Bestnu}
|B|\le \sqrt{d}\frac{(d-1)}{2}(a_V+\nu^\alpha)
\ee where $|B|$ denotes the Frobenius norm of the symmetric matrix $B$.

We denote by $b\mbE(\nu,a_V,a_W,\bV,\bW,V_0,\nabla V_0,T,d)$ the expression on the right-hand side of \eqref{bound-Enuaalpha}, by $b\chi(\nu,a_V,a_W,\bV,\bW,V_0,\nabla V_0,T,d)$ the expression on the right-side of \eqref{bound-chialphanu}, respectively by $bB(\nu,a_V,d)$ the expression on the right-hand side of \eqref{Bestnu}. Let us note that since $a_V,a_W>0$ the functions $b\mbE,b\chi$ and $bB$ are increasing as functions in $\nu$. Thus denoting  the balls:

\bea
\mathfrak{B}(\nu_0,a_V,a_W,\bV,\bW,V_0,\nabla V_0,T,d):=\{&(\mbE,\chi,B)\in L^2\times L^2\times L^\infty);\non\\
&\|E\|_{L^2}\le b\mbE(\nu_0,a_V,a_W,\bV,\bW,V_0,\nabla V_0,T,d),\non\\
&\|\chi\|_{L^2}\le b\chi(\nu_0,a_V,a_W,\bV,\bW,V_0,\nabla V_0,T,d),\non\\
&|B|_{L^\infty}\le bB(\nu_0,a_V,d)\}
\eea (where the spaces $L^2$, respectively $L^\infty$ are taken to be both in space and time) we have that $\mathfrak{B}(\nu,a_V,a_W,\bV,\bW,V_0,\nabla V_0,T,d)\subset \mathfrak{B}(\nu_0,a_V,a_W,\bV,\bW,V_0,\nabla V_0,T,d)$ for all $\nu\le \nu_0$.

We consider the metric space \footnote{The space also depends on several other variables, as indicated in the definition of $\mathfrak{B}$, but, for the sake of readability, we omit explicitly indicating notationally this dependence. } $$\widetilde{\mathcal{EB}\Psi_{2,\infty,2}}^w(\nu_0):=\cup_{\nu\ge 0}\mathcal{EB}\Psi_{2,\infty,2}^\nu\cap \mathfrak{B}(\nu_0,a_V,a_W,\bV,\bW,V_0,\nabla V_0,T,d)$$ which is now endowed with the weak topology of $L^2_T(L^2)$ for the $E$ and $\chi$ components and with the 
weak-* topology of $L^\infty_T(L^\infty)$ for the $B$  component. It is known that since we are on a bounded set this topology is metrizable so we will consider without further comment the space as being a metric space. This is useful as it allows us to use the sequential characterisation of the $\Gamma$-limit. We have:

\begin{proposition}Let $V_0\in L^2(\Omega,\mathbb{R}^d)$ be 
a divergence-free vector field 
of zero spatial mean with $\nabla V_0\in L^2(\Omega,\mathbb{R}^{d\times d}) $ and $\nu_0>0$ some arbitrary positive number. Then 

\be
\widetilde{\mathcal{A}^\nu_{NS}}\stackrel{\Gamma}{\longrightarrow} \widetilde{\mathcal{A}^0_{NS}}
\ee  in $\widetilde{\mathcal{EB}\Psi_{2,\infty,2}}^w(\nu_0)$.
\end{proposition}

We omit the proof as it follows closely the one of \eqref{prop:gammaconvstrong}, namely the liminf part is the same, because we worked with weak convergence. Also the recovery sequence in the limsup part can be chosen exactly the same.

Then, we have:

\begin{theorem}
Let $(\mbE^\nu,\chi^\nu,B^\nu)$ be minimisers of $\widetilde{\mathcal{A}^\nu_{NS}}$. Then, there exists a sequence $\{(\mbE^{\nu_k},\chi^{\nu_k},B^{\nu_k})\}_{k\in\N}$ and a minimiser $(\mbE^0,\chi^0,B^0)$ of $\widetilde{\mathcal{A}^0_{NS}}$ such that as $k\to\infty$ we have $(\mbE^{\nu_k},\chi^{\nu_k},B^{\nu_k})\to (\mbE^0,\chi^0,B^0)$ in $\widetilde{\mathcal{EB}\Psi_{2,\infty,2}}^w(\nu_0)$.
\end{theorem}

The proof is standard, as we can apply the Hahn Banach Theorem and extract a subsequence converging in the corresponding weak topology. Using the liminf and limsup properties of the $\Gamma$-limit one can check that the limit of the sequence is a minimiser for $\widetilde{\mathcal{A}^0_{NS}}$.

\bigskip
\textbf{Acknowledgment:} A.A. has been partially supported by Simons Pivot Fellowship grant \# 983171. B.S. has been partially supported by the project STAR PLUS 2020 Linea 1 (21-UNINA-EPIG-172) New perspectives in the Variational modeling of Continuum Mechanics. A.Z. has been partially supported by the Basque Government through the BERC 2022-2025 program and by the Spanish State Research Agency through BCAM Severo Ochoa CEX2021-001142 and through Grant PID2023-146764NB-I00 funded by MICIU/AEI /10.13039/501100011033 and by ERDF, EU. A.Z. was also partially supported  by a grant of the Ministry of Research, Innovation and Digitization, CNCS - UEFISCDI, project number PN-III-P4-PCE-2021-0921, within PNCDI III.

All the authors acknowledge the hospitality in Hausdorff Research Institute for Mathematics funded by the Deutsche
Forschungsgemeinschaft (DFG, German Research Foundation) under Germany's Excellence Strategy EXC-2047/1–390685813.

\bibliographystyle{alpha}
\bibliography{dualNSE.bib}

\end{document}